\documentclass[a4paper,leqno]{amsart}

\usepackage{amsmath}
\usepackage{amssymb}
\usepackage{amsfonts}
\usepackage{enumerate}
\usepackage{amsthm}
\usepackage{esint}
\usepackage{bbm}
\usepackage{tabularx}
\usepackage{graphicx}
\usepackage{epstopdf}

\newtheorem{teo}{Theorem}[section]
\newtheorem{lem}[teo]{Lemma}
\newtheorem{cor}[teo]{Corollary}
\newtheorem{prop}[teo]{Proposition}

\theoremstyle{definition}
\newtheorem{defin}[teo]{Definition}
\newtheorem{oss}[teo]{Remark}

\newtheorem{assumption}{Assumption}
\newtheorem{claim}{Claim}

\renewcommand{\eqref}[1]{\textnormal{(\ref{#1})}}

\numberwithin{equation}{section}

\newcommand{\R}{\mathbb{R}}
\newcommand{\N}{\mathbb{N}}

\title[Optimality for multiscale decompositions]{On the optimality of convergence conditions for multiscale decompositions in imaging and inverse problems}

\author[S.~Rebegoldi]{Simone Rebegoldi}
\author[L.~Rondi]{Luca Rondi}

\address[S.~Rebegoldi]{Dipartimento di Scienze Fisiche, Informatiche e Matematiche, Universit\`a degli Studi di Modena e Reggio Emilia, Via Campi 213/B, 41125 Modena, Italy}
\email{simone.rebegoldi@unimore.it}

\address[L.~Rondi]{Dipartimento di Matematica, Universit\`a degli Studi di Pavia, Via Ferrata 5,
27100 Pavia, Italy}
\email{luca.rondi@unipv.it}

\date{}

\begin{document}

\begin{abstract}
We consider the multiscale procedure developed by Modin, Nachman and Rondi, Adv. Math. (2019), for inverse problems, which was inspired by the multiscale decomposition of images by Tadmor, Nezzar and Vese, Multiscale Model. Simul. (2004).
We investigate under which assumptions this classical procedure
is enough to have convergence in the unknowns space without resorting to use the tighter multiscale procedure from the same paper. We show that this is the case for linear inverse problems  when the regularization is given by the norm of a Hilbert space. Moreover, in this setting the multiscale procedure improves the stability of the reconstruction. On the other hand, we show that, for the classical multiscale procedure, convergence in the unknowns space might fail even for the linear case with a Banach norm as regularization.

\bigskip

% 65J22  	Numerical solution to inverse problems in abstract spaces

\medskip

\end{abstract}

\maketitle

\setcounter{section}{0}
\setcounter{secnumdepth}{2}

\section{Introduction}

We consider the multiscale decomposition developed in \cite{M-N-R}, which extended to inverse problems and other imaging applications the
so-called $(BV$-$L^2)$ multiscale decomposition of images introduced in \cite{T-N-V} and based on the classical ROF model for denoising \cite{R-O-F}.
We briefly review this general multiscale procedure and its main features.

Let $X$ be a real Banach space with norm $\|\cdot\|=\|\cdot\|_X$. Let $E$ be a
closed nonempty subset of $X$.

Let $Y$ be a metric space with distance $d=d_Y$. Let $\tilde{\Lambda}\in Y$ and $\Lambda:E\to Y$ such that
 the function $E\ni \sigma\mapsto d(\tilde{\Lambda},\Lambda(\sigma))$  is continuous with respect to the (strong) convergence in $X$.

The model we have in mind is the following kind of inverse problem. We aim to recover the unknown properties of a physical system by some kind of indirect measurements, for example by
 suitably probing it and measuring its reaction. In this case
$X$ is the \emph{unknown data space} or \emph{unknowns space}, with $E$ representing the admissible ones, and $Y$ is the \emph{measurements space}. The possibly nonlinear operator $\Lambda$ is the operator mapping the unknown data into the corresponding measurements. We aim to recover the unknown data $\tilde{\sigma}\in E$ from $\tilde{\Lambda}$, an approximation of the corresponding measurements $\Lambda(\tilde{\sigma})$. Even if $\Lambda$ is injective, that is, we have uniqueness for the inverse problem, usually its inverse
is not continuous, that is, we lack stability of the inverse problem. In order to recover stability, we have to impose suitable a priori assumptions on the unknowns or, from the numerical point of view, adopt a suitable regularization procedure. Since here we use a variational method of reconstruction, we focus on a Tikhonov type regularization with a regularization operator $R$.

Let us fix $R:X\to [0,+\infty]$, with $R(0)=0$, and two positive constants $\alpha$ and $\beta$ such that the following holds.
\begin{assumption}\label{existenceassum}
For $\hat{\sigma}=0$ and for any $\hat{\sigma}\in E$ and any $\lambda>0$ the following minimization problem admits a solution
\begin{equation}\label{existenceminimizereq}
\min\left\{\lambda\big(d(\tilde{\Lambda},\Lambda(\hat{\sigma}+u))\big)^{\alpha}+(R(u))^{\beta}:\ \hat{\sigma}+u\in E\right\}.
\end{equation}
\end{assumption}

Inspired by \cite{T-N-V}, in \cite{M-N-R} the following \emph{multiscale procedure} was developed.
Fixed positive parameters $\lambda_n$, $n\geq 0$, 
let $\sigma_0=u_0\in E$ solve
\begin{equation}\label{regularizedpbm}
\min\left\{\lambda_0\big(d(\tilde{\Lambda},\Lambda(u))\big)^{\alpha}+(R(u))^{\beta}:\ u\in E\right\}.
\end{equation}
Then by induction we define $\sigma_n$, $n\geq 1$, as follows. Let $u_n$ be a solution to
 \begin{equation}\label{inductiveconstr}\min\left\{\lambda_n\big(d(\tilde{\Lambda},\Lambda(\sigma_{n-1}+u))\big)^{\alpha}+(R(u))^{\beta}:\ \sigma_{n-1}+u\in E\right\}\end{equation}
and we denote $\sigma_n=\sigma_{n-1}+u_n$, that is,
 $\sigma_n$ is the partial sum
 \begin{equation}\label{partialsum1}
 \sigma_n=\sum_{j=0}^nu_j\quad\text{for any }n\in\mathbb{N}.
 \end{equation}
 By Assumption~\ref{existenceassum}, the sequence $\{\sigma_n\}_{n\geq 0}$ exists, but it might be not uniquely determined.
Since 
\begin{equation}
d(\tilde{\Lambda},\Lambda(\sigma_n))\leq d(\tilde{\Lambda},\Lambda(\sigma_{n-1}))\quad\text{for any }n\geq 1,
\end{equation}
we can define
$$\varepsilon_0=\lim_nd(\tilde{\Lambda},\Lambda(\sigma_n))$$
and note that
$$\varepsilon_0\geq \delta_0=\inf\big\{d(\tilde{\Lambda},\Lambda(\sigma)):\ \sigma\in E\big\}.$$

The following two theoretical questions are of interest.
\begin{enumerate}[1)]
\item \textbf{Convergence in the measurements space.}  Prove that $\{\sigma_n\}_{n\geq 0}$ is a minimizing sequence, that is,
$\varepsilon_0=\delta_0$ or
$$\lim_nd(\tilde{\Lambda},\Lambda(\sigma_n))=\inf\big\{d(\tilde{\Lambda},\Lambda(\sigma)):\ \sigma\in E\big\}.$$
\item \textbf{Convergence in the unknowns space.} If 1) holds, prove that $\sigma_n$ converges in $X$, possibly up to a subsequence, to
$\sigma_{\infty}\in E$.
Moreover, show that $\sigma_{\infty}$ is a solution to the inverse problem, namely
it solves the minimization problem
\begin{equation}\label{minexist}
\min\big\{d(\tilde{\Lambda},\Lambda(\sigma)):\ \sigma\in E\big\}.
\end{equation}
%$$d(\tilde{\Lambda},\Lambda(\sigma_{\infty}))=\min\big\{d(\tilde{\Lambda},\Lambda(\sigma)):\ \sigma\in E\big\}.$$
\end{enumerate}

We assume that 1) holds. Indeed, if $\sigma_{n_k}\to \sigma_{\infty}$ in $X$ as $k\to +\infty$, then $\sigma_{\infty}$
solves
\eqref{minexist}.
Therefore, a necessary condition for 2) to hold is that the minimization problem \eqref{minexist} admits a solution.
%that
%$$d(\tilde{\Lambda},\Lambda(\sigma_{\infty}))=\min\big\{d(\tilde{\Lambda},\Lambda(\sigma)):\ \sigma\in E\big\}.$$
%\end{enumerate}
%
%We assume that 1) holds. Indeed, if $\sigma_{n_k}\to \sigma_{\infty}$ in $X$ as $k\to +\infty$, then $\sigma_{\infty}$
%is a solution to the inverse problem, that is, it solves the minimization problem
%\begin{equation}\label{minexist}
%\min\big\{d(\tilde{\Lambda},\Lambda(\sigma)):\ \sigma\in E\big\}.
%\end{equation}
%Therefore, a necessary condition for 2) to hold is that the minimization problem \eqref{minexist} admits a solution.

Before discussing questions 1) and 2), let us review the development of the theory so far.
The multiscale decomposition started in a beautiful paper, \cite{T-N-V}, where a noisy image $f$ was decomposed by the same iterative procedure using the ROF model for denoising, \cite{R-O-F},  namely for $f\in L^2(\Omega)$
$$\min\left\{\lambda_0\|f- u\|^{2}_{L^2(\Omega)}+TV(u):\ u\in L^2(\Omega)\right\}.$$
In other words, here $X=E=Y=L^2(\Omega)$, $\Lambda$ is equal to the identity, $R(u)=TV(u)$, $TV$ denoting the total variation, and $\Omega$ is either $\mathbb{R}^2$ or a bounded Lipschitz domain contained in $\R^2$. In \cite{T-N-V} it was proved that for $f\in BV(\Omega)$, or some intermediate space between $BV(\Omega)$ and $L^2(\Omega)$, we have the following multiscale decomposition
$$f=\lim_n\sigma_n=\sum_{n=0}^{+\infty}u_n\quad\text{in }L^2(\Omega).$$
In \cite{T-N-V-2}, such a multiscale decomposition method was applied, without any convergence results, to other imaging problems, for instance to a deblurring problem. In this case, $\Lambda:L^2(\Omega)\to L^2(\Omega)$ is a linear blurring operator (for example the convolution with a suitable point spread function if $\Omega=\R^2)$, and $f$ is a perturbation due to noise of a blurred original image
$\hat{\sigma}$, that is, $f=\Lambda(\hat{\sigma})+\text{noise}$ and the multiscale procedure is based on
$$\min\left\{\lambda_0\|f- \Lambda(u)\|_{L^2(\Omega)}^{2}+TV(u):\ u\in L^2(\Omega)\right\}.$$
The deblurring problem is the prototype for the linear case we discuss below.

In \cite{M-N-R}, the convergence analysis for these multiscale procedures was greatly enhanced and the multiscale theory was extended to include general nonlinear inverse problems and other imaging applications such as image registration in the LDDMM framework. The main results of \cite{M-N-R}, at least for inverse problems, will be summarized below. 
An immediate side effect of \cite[Theorem~2.1]{M-N-R}, which is here recalled as Theorem~\ref{1thm},
was the proof that the multiscale decomposition of images in \cite{T-N-V} is valid for any $f\in L^2(\Omega)$. We point out that such a result was actually announced in \cite{Jaw-Mil} but their proof, which used completely different methods, has never been published.

After \cite{M-N-R}, multiscale procedures gained more and more attention and several further developments appeared. 
In \cite{L-R-V} quantitative convergence and error estimates as well as stopping criteria depending on noise level were analyzed for the linear case. Still in the linear case, in 
\cite{K-R-W}, the use of convex or nonconvex regularization terms was discussed.
In \cite{M-N}, the procedure has been extended to optimal transport problems. In \cite{Wolf} the multiscale procedure has been used for solving a blind deconvolution problem.

Back to our questions,
about issue 1), a comprehensive solution is given in \cite[Theorem~2.1]{M-N-R}, which we recall for the convenience of the reader.
\begin{teo}\label{1thm}
Assume that
\begin{equation}\label{structcondition}
R(0)=0,\ R(-u)=R(u)\text{ and }R(u+\tilde{u})\leq R(u)+R(\tilde{u})\text{ for any }u,\,\tilde{u}\in X.
\end{equation}
 Assume also that
\begin{equation}\label{densecondition}
\{u\in E:\ R(u)<+\infty\}\text{ is dense in }E
\end{equation}
and that Assumption 1 holds.

If
\begin{equation}\label{firstcondcoeff}
\limsup_n \frac{2^{\beta n}}{\lambda_n}<+\infty,\end{equation}
then for the multiscale sequence $\{\sigma_n\}_{n\geq 0}$ given by \eqref{regularizedpbm}, \eqref{inductiveconstr} and \eqref{partialsum1}
we have $\varepsilon_0=\delta_0$, that is,
$$\lim_nd(\tilde{\Lambda},\Lambda(\sigma_n))=\inf\big\{d(\tilde{\Lambda},\Lambda(\sigma)):\ \sigma\in E\big\}.$$
\end{teo}

The following assumptions on $R$ guarantee that the required hypothesis on $R$ in Theorem~\ref{1thm} are satisfied.
\begin{assumption}\label{strongerassum} We assume that $R$ satisfies
\eqref{structcondition}, \eqref{densecondition},
\begin{equation}\label{compactnesscondition}
\{u\in X:\ R(u)\leq b\}\text{ is relatively compact in }X,\text{ for any }b\in \mathbb{R}
\end{equation}
and
\begin{equation}\label{lsccondition}
R\text{ is lower semicontinuous on }X,\text{ with respect to the convergence in }X.
\end{equation}
\end{assumption}

Issue 2) is much more delicate. Given Assumption~\ref{strongerassum}, in \cite{M-N-R} a so-called \emph{tighter multiscale procedure} was developed, for which a result analogous to Theorem~\ref{1thm} holds, see
\cite[Theorem~2.4]{M-N-R}. For such a tighter version, in \cite[Theorem~2.5]{M-N-R} convergence, up to a subsequence, of $\sigma_n$ in $X$ was proved, provided the following is satisfied.

\begin{assumption}\label{stronger}
There exists $\hat{\sigma}\in E$ such that
\begin{equation}\label{anothermin00}
d(\tilde{\Lambda},\Lambda(\hat{\sigma}))
=\delta_0=\min\big\{d(\tilde{\Lambda},\Lambda(\sigma)):\ \sigma\in E\big\}\quad\text{and}\quad R(\hat{\sigma})<+\infty.
\end{equation} 
Without loss of generality, one can further assume that $\hat{\sigma}$ solves the following minimization problem
\begin{equation}\label{minimal00}
\min\big\{ R(\sigma):\ \sigma\in E\text{ and }d(\tilde{\Lambda},\Lambda(\sigma))=\delta_0\big\}=R_0<+\infty.
\end{equation}
\end{assumption}

Under Assumption~\ref{stronger}, we call set of \emph{optimal solutions} of the inverse problem the set $S_{opt}$ given by
$$S_{opt}=\big\{\sigma\in E:\ d(\tilde{\Lambda},\Lambda(\sigma))=\delta_0\text{ and }R(\sigma)=R_0\big\}.$$

\begin{oss}If Assumption~\ref{stronger} holds, then we can drop in Theorem~\ref{1thm} the denseness condition \eqref{densecondition}.
\end{oss}

In this paper we investigate whether, under Assumption~\ref{stronger}, the tighter procedure is really needed or we can guarantee convergence of $\sigma_n$ also for the classical multiscale procedure. We consider two particular cases.

\begin{enumerate}[A)]
\item \textbf{The linear case.}
Let $X$ and $Y$ be Banach spaces and $E=X$. Let $\Lambda:X\to Y$ be a continuous linear operator.
Let $R:X\to [0,+\infty]$ be a convex function on $X$.
We assume that $\alpha\geq 1$ and $\beta\geq 1$.

\item \textbf{The convex case.}
Let $X$ be a Banach space and $Y$ be a metric space.
We assume that $E$ is convex and that $E\ni\sigma\to d(\tilde{\Lambda},\Lambda(\sigma))$ is a continuous convex function.
Let $R:X\to [0,+\infty]$ be a convex function on $X$.
We assume that $\alpha\geq 1$ and $\beta\geq 1$.
\end{enumerate}

We note that the linear case is a particular case of the convex one. %For the linear case, in particular, we shall focus on the case when $Y=H$ is a Hilbert space, $\alpha=2$ and $\beta=1$.
For the convex case, an interesting assumption on $R$ is the following.
\begin{assumption}\label{strictconvexity}
For any $u$ and $\tilde{u}\in X$ such that $u\neq \tilde{u}$ and $R(u)=R(\tilde{u})<+\infty$, we have
$$R(u+t(\tilde{u}-u))<R(u)=R(\tilde{u})\quad\text{for any }t\in (0,1).$$
\end{assumption}

Moreover, we talk of a \emph{Banach norm regularization} if $R$ is of the following kind.
\begin{assumption}\label{Banachassum}
Let $F$ be a Banach space such that 
$\tilde{E}=F\cap X\neq \{0\}$. Then we define
the regularization $R$ for any $u\in X$ as
\begin{equation}\label{banachregular}
R(u)=\left\{\begin{array}{ll}
\|u\|_{F} &\text{if }u\in F\cap X\\
+\infty &\text{otherwise}.
\end{array}\right.
\end{equation}
We note that $R$ satisfies \eqref{structcondition} and it is convex on $\tilde{E}$, thus on $X$.
\end{assumption}

\begin{oss}
We note that if $R$ is the norm of a vector space, such as in Assumption~\ref{Banachassum} for example, Assumption~\ref{strictconvexity} correspond to strict convexity of the normed vector space.
\end{oss}

Finally, we talk of a \emph{Hilbert norm regularization} if $R$ is
as in Assumption~\ref{Banachassum} with $F$ being a Hilbert space, namely if $R$ is of the following kind.
\begin{assumption}\label{hilbertassum}
 Let $H_1$ be a Hilbert space such that $\tilde{E}=H_1\cap X\neq \{0\}$. Then we define
the regularization $R$ for any $u\in X$ as
\begin{equation}\label{hilbertregular}
R(u)=\left\{\begin{array}{ll}
\|u\|_{H_1} &\text{if }u\in H_1\cap X\\
+\infty &\text{otherwise}.
\end{array}\right.
\end{equation}
We note that $R$ satisfies \eqref{structcondition}, it is convex on $\tilde{E}$, thus on $X$, and it satisfies Assumption~\ref{strictconvexity}.
\end{assumption}

About question 2), that is, convergence in the unknowns space for the classical multiscale procedure, 
we have results both of positive kind (Section~\ref{Hilbertregnormsec}) and of negative kind (Section~\ref{examplesec}).

In Section~\ref{convergressec} we show our convergence results for the classical multiscale procedure. Under Assumption~\ref{stronger}, convergence up to subsequences is proved in the convex case when $R$ is a Hilbert norm regularization as in Assumption~\ref{hilbertassum}, see 
Theorem~\ref{hilbertregthm}.
% In Section~\ref{convergressec} we show
%convergence results. Namely, under Assumption~\ref{stronger}, the answer to our query is positive in a special case, namely in the convex case when
%$R$ is a Hilbert norm regularization as in Assumption~\ref{hilbertassum}, see 
%Theorem~\ref{hilbertregthm}.
Another interesting feature of the Hilbert norm regularization is that, under Assumption~\ref{stronger}, convergence, both in the measurements space and in the unknowns space, is guaranteed provided $\displaystyle \lim_n\lambda_n=+\infty$ without any further control of the rate with which the sequence goes to $+\infty$, that is, we can drop the assumption \eqref{firstcondcoeff}.
Let us note that under the assumptions of Theorem~\ref{hilbertregthm}, there exists exactly one optimal solution, that is, $S_{opt}=\{\hat{\sigma}\}$, with $\hat{\sigma}$ solving \eqref{minimal00}.
In the linear case, still for a Hilbert norm regularization, something more can be said, Corollary~\ref{linearcase}.
In fact, we have that
$\displaystyle \lim_n \sigma_n=\hat{\sigma}$ in $X$, that is, the following multiscale decomposition holds
$$\hat{\sigma}=\sum_{n=0}^{+\infty}u_n\quad\text{in }X.$$
% In Section~\ref{convergressec} we show
%convergence results. Namely, under Assumption~\ref{stronger}, the answer to our query is positive in a special case, namely in the convex case when
%$R$ is a Hilbert norm regularization as in Assumption~\ref{hilbertassum}, see 
%Theorem~\ref{hilbertregthm}.
%Let us note that under the assumptions of Theorem~\ref{hilbertregthm}, there exists exactly one optimal solution, that is, $S_{opt}=\{\hat{\sigma}\}$, with $\hat{\sigma}$ solving \eqref{minimal00}.
%Another interesting feature of the Hilbert norm regularization is that, under Assumption~\ref{stronger}, convergence, both in the measurements space and in the unknowns space, is guaranteed provided $\displaystyle \lim_n\lambda_n=+\infty$ without any further control of the rate with which the sequence goes to $+\infty$, that is, we can drop the assumption \eqref{firstcondcoeff}. In the linear case, still for a Hilbert norm regularization, something more can be said, Corollary~\ref{linearcase}.
%In fact, we have that
%$\displaystyle \lim_n \sigma_n=\hat{\sigma}$ in $X$, that is, the following multiscale decomposition holds
%$$\hat{\sigma}=\sum_{n=0}^{+\infty}u_n\quad\text{in }X.$$

We also show that for the Hilbert norm regularization, the multiscale procedure helps to improve the stability of the reconstruction for inverse ill-posed problems, with respect to the single-step regularization, see Section~\ref{stabsec}. A numerical illustration and validation of such an improvement of the stability for an image deblurring problem is contained in Section~\ref{numerics}. 

Unfortunately, such results might cease to hold even in the linear case for a Banach norm regularization and under Assumption~\ref{stronger}. Actually, convergence in the unknowns space may fail completely since $\|\sigma_n\|_X$ might be even diverging to $+\infty$. In order to show this, 
in Section~\ref{examplesec} we develop several counterexamples. In  Theorems~\ref{examplethm} and \ref{examplethmsecond} we consider two examples in the discrete case, for spaces of sequences. In Theorem~\ref{examplethmthird} we consider a continuous case, with the same structure as the original construction developed in \cite{T-N-V} on the line, that is, for the multiscale decomposition of signals, but with a linear continuous and injective operator $\Lambda$ which is different from the identity. In other words, for some very nasty linear blurring or deforming operators of signals, the multiscale procedure might not be converging in the unknowns space. We conclude that to guarantee convergence of the classical multiscale procedure some further structural assumptions on the problem or the operator $\Lambda$ are required. We show here that the regularization being the norm of a Hilbert norm suffices. Another interesting example is the so-called \emph{closure bound}, \cite[formula (4.19)]{T}, used in \cite{T} to construct through the multiscale procedure bounded solutions in $\R^N$ to the equation $\mathrm{div}(u)=f$ for $f\in L^N$.

The plan of the paper is as follows.
Section~\ref{linearsec} is a preliminary section where we review the linear case, when $Y=H$ is a Hilbert space, $\alpha=2$ and $\beta=1$. We show a Parseval-like identity (Section~\ref{Parssec}), following Meyer's work \cite{YMey}
we provide a characterization of minimizers (Section~\ref{mincarsec}), and we briefly discuss applications to the classical $(BV\text{-}L^2)$-decomposition of images or signals developed in \cite{T-N-V} (Section~\ref{BVL2sec}).
In Section~\ref{Hilbertregnormsec}, we develop our positive results for the Hilbert norm regularization.
Finally, in Section~\ref{examplesec}, we construct several counterexamples showing the possible failure of convergence in the unknowns space.

\section{The linear case}\label{linearsec}
In this section we review some basic properties of the linear case, when $Y$ is a Hilbert space. Apart from some remarks in Section~\ref{BVL2sec}, all other results in this section are well-known.

Let $Y=H$ be a Hilbert space with scalar product $\langle\cdot,\cdot\rangle_H$ and norm $\|\cdot\|_H$. Let $X$ be a Banach space, $E=X$, and let $\Lambda:X\to H$ be a bounded linear operator.
Let $R:X\to [0,+\infty]$ be a function such that $R(0)<+\infty$.

We consider the following minimization problem, for some positive parameter $\lambda_0$ and a given datum $f\in H$,
\begin{equation}\label{minpbm}
\min\left\{\lambda_0\|f-\Lambda(u)\|_H^2+R(u):\ u\in X  \right\}.
\end{equation}

\begin{oss} The minimum, if it exists, is always finite since for $u=0$ the value of the functional to be minimized is $\|f\|_H^2+R(0)$. Hence, if $u_0$ is a minimizer,  then $R(u_0)<+\infty$.
\end{oss}

We begin with a remark on uniqueness, whose proof is standard.

\begin{oss}\label{uniquenessremark}
Assume that $R$ is convex on $X$.
If a solution $u_0$ to \eqref{minpbm} exists, then
$$u\text{ solves \eqref{minpbm} if and only if }\Lambda(u-u_0)=0\text{ and }R(u)=R(u_0).$$
Thus if $\Lambda$ is injective on
$\tilde{E}=\{u\in X:\ R(u)<+\infty\}$
 or $R$ satisfies Assumption~\ref{strictconvexity}, then uniqueness holds.
\end{oss}

\subsection{Parseval-like identity}\label{Parssec}

We call $\tilde{E}=\{u\in X:\ R(u)<+\infty\}$. From now on, we assume that $\tilde{E}$ is a linear subspace of $X$ and that $R$ is even and positively $1$-homogeneous on $\tilde{E}$, that is,
$R(au)=|a|R(u)$ for any $u\in \tilde{E}$ and any $a\in \R$. We note that $R(0)=0$.

\begin{oss} If we 
further assume that $R$ satisfies the triangle inequality on $\tilde{E}$, that is,
$$R(u+\tilde{u})\leq R(u)+R(\tilde{u})\quad\text{for any }u,\, \tilde{u}\in \tilde{E},$$
then $R$ satisfies \eqref{structcondition} and it is convex on $\tilde{E}$, thus on $X$.
\end{oss}

We begin with the following statement.

\begin{prop} Let $u_0$ be a solution to \eqref{minpbm}. We call $v_0=f-\Lambda(u_0)$, that is, $f=\Lambda(u_0)+v_0$.
Then
\begin{equation}\label{firstproperty}
\langle v_0,\Lambda(u_0)\rangle_{H}=\frac{1}{2\lambda_0}R(u_0).
\end{equation}
Consequently
\begin{equation}\label{secondproperty}
\|f\|_H^2=\|\Lambda(u_0)\|_H^2+\|v_0\|_H^2+\frac{1}{\lambda_0}R(u_0).
\end{equation}
\end{prop}

\begin{proof} We have that $u_0$ is a minimizer if and only if
\begin{multline}\label{charmin}
\lambda_0\|v_0\|_H^2+R(u_0)\leq \lambda_0\|v_0-\varepsilon \Lambda(h)\|_H^2+R(u_0+\varepsilon h)\\\quad\text{for any $h\in \tilde{E}$ and any $\varepsilon\in\R$}.
\end{multline}

Let us apply the formula to $h=u_0$ itself, with $|\varepsilon|<1$. We obtain
$$2\varepsilon\lambda_0\langle v_0,\Lambda(u_0)\rangle_{H}\leq \varepsilon R(u_0)+\varepsilon^2\lambda_0\| \Lambda(u_0)\|_H^2.$$
Using $0<\varepsilon<1$, dividing by $\varepsilon$ and letting $\varepsilon$ go to $0$, we deduce that
$$2\lambda_0\langle v_0,\Lambda(u_0)\rangle_{H}\leq R(u_0).$$
The converse inequality is obtained analogously using $-1<\varepsilon<0$.\end{proof}

Let us assume that a solution to \eqref{minpbm} exists for any $f\in H$. Then we can build the multiscale procedure, with a sequence of positive  parameters $\lambda_j$, $j\geq 0$.
We call $v_{-1}=f$. Then, if $u_j$ is the solution at step $j$, and $v_j=v_{j-1}-\Lambda(u_j)$, for any $j\geq 0$, we conclude that for any $n\geq 0$
$$f=\sum_{j=0}^n\Lambda(u_j)+v_n=\Lambda\left(\sum_{j=0}^nu_j\right)+v_n$$
and
\begin{equation}\label{thirdproperty}
\|f\|_H^2=\sum_{j=0}^n\left(\|\Lambda(u_j)\|_H^2+\frac{1}{\lambda_j}R(u_j)\right)+\|v_n\|_H^2.
\end{equation}

Let $$\delta_0=\inf\left\{\|f-\Lambda(\sigma)\|_H:\ \sigma\in X\right\}.$$ Then
$$f=\tilde{f}+\tilde{v}$$
where $\tilde{f}\in\overline{\Lambda(X)}$ and $\tilde{v}$ is orthogonal to $\overline{\Lambda(X)}$, thus $\|\tilde{v}\|_H=\delta_0$.
Thus $v_n=\tilde{v}+\tilde{v}_n$ where $ \tilde{v}_n\in \overline{\Lambda(X)}$ and
$$\|v_n\|^2=\|\tilde{v}\|_H^2+\|\tilde{v}_n\|_H^2=\delta_0^2+\|\tilde{v}_n\|_H^2.$$

We can conclude with the following proposition, containing the Parseval-like identity.
\begin{prop}\label{multprop}
Let us assume that a solution to \eqref{minpbm} exists for any $f\in H$.
If $\tilde{v}_n\to 0$ in $H$, that is, $\displaystyle \lim_n\|f-\Lambda(\sigma_n)\|_H=\delta_0$, then
\begin{equation}\label{multiscale+pars}
f=\tilde{v}+\sum_{j=0}^{+\infty}\Lambda(u_j)\quad\text{and}\quad\|f\|_H^2=\|\tilde{v}\|^2_H+\sum_{j=0}^{+\infty}\left(\|\Lambda(u_j)\|_H^2+\frac{1}{\lambda_j}R(u_j)\right).
\end{equation}
Here the series on the left has to be intended in $H$ and we recall that $\|\tilde{v}\|_H=\delta_0$.
\end{prop}

\subsection{Characterization of minimizers}\label{mincarsec}

Let $R$ be a Banach norm regularization, that is, it satisfies Assumption~\ref{Banachassum} for a Banach space $F$.
We recall that $R$ satisfies \eqref{structcondition} and it is convex on $\tilde{E}$, thus on $X$. Moreover, it is even and positively $1$-homogeneous on $\tilde{E}$.

Let $G=F^{\ast}$ and let $\|\cdot\|_{\ast}$ be its norm.
Let $f\in H$ and $f^{\ast}$ be the linear functional on $H$ associated to $f$.
We say that $f^{\ast}\circ \Lambda$ belongs to $G$ if the functional $F\cap X \ni u\mapsto \langle f,\Lambda(u)\rangle_{H}$ is bounded with respect to the $F$-norm on $u$, that is, there exists $C\geq 0$ such that
$$\langle f,\Lambda(u)\rangle_{H}\leq C\|u\|_F\quad\text{for any }u\in F\cap X.$$
We call
\begin{equation}\label{starnorm}
\|f^{\ast}\circ \Lambda\|_{\ast}=\sup\{\langle f,\Lambda(u)\rangle_{H}:\ u\in F\cap X\text{ with }\|u\|_F=1\}.
\end{equation}
Then, by Hahn-Banach, $f^{\ast}\circ \Lambda|_{F\cap X}$ can be extended to an element of $G$ with the same $\|\cdot\|_{\ast}$-norm.

Using Meyer's arguments in \cite{YMey}, the following well-known result can be easily obtained.  

\begin{prop}\label{charprop}
Let $f\in H$ and let $u_0$ be a solution to \eqref{minpbm}.
Then we have the following characterizations.
\begin{enumerate}[a\textnormal{)}]
\item $u_0$ is a minimizer if and only if 
\begin{equation}\label{char}
\|v_0^{\ast}\circ\Lambda\|_{\ast}\leq \dfrac1{2\lambda_0}\quad\text{and}\quad\langle v_0,\Lambda(u_0)\rangle_{H}=\dfrac1{2\lambda_0}\|u_0\|_F,\end{equation}
where $v_0=f-\Lambda(u_0)$ and $v_0^{\ast}$ is the linear functional on $H$ associated to $v_0$.
\item $u_0=0$ if and only if $f^{\ast}\circ \Lambda\in G$ and $\|f^{\ast}\circ \Lambda\|_{\ast}\leq \dfrac1{2\lambda_0}.$
\item If $\|f^{\ast}\circ \Lambda\|_{\ast}> \dfrac1{2\lambda_0}$, including when $f\circ\Lambda\notin G$, then \eqref{char} may be replaced by
\begin{equation}\label{char2}
\|v_0^{\ast}\circ\Lambda\|_{\ast}=\dfrac1{2\lambda_0}\quad\text{and}\quad\langle v_0,\Lambda(u_0)\rangle_{H}=\dfrac1{2\lambda_0}\|u_0\|_F>0.\end{equation}
\end{enumerate}
\end{prop} 

\subsection{Applications to the classical ($BV$-$L^2$) decompositions}\label{BVL2sec}

In the classical ($BV$-$L^2$) decomposition we have $X=E=H=L^2(\Omega)$, $\Lambda$ is just the identity map,  $\alpha=2$ and $\beta=1$. The set $\Omega$ is either $\R^N$, $N\geq 1$,
or a bounded domain with Lipschitz boundary contained in $\mathbb{R}^N$, $N\geq 1$. About $R$, this is a $BV$-related norm or seminorm, typically it is the total variation.

For $\Omega=\mathbb{R}^N$, $N\geq 1$,
the homogeneous $BV$ space on  $\mathbb{R}^N$, $N\geq 1$, is defined as follows.
We say that $u\in BV(\mathbb{R}^N)$ if $Du$, its gradient in the distributional sense, is a bounded vector valued Radon measure on $\mathbb{R}^N$ and $u$ satisfies a suitable condition at infinity. Namely,
if $N=1$, then $BV(\mathbb{R})\subset L^{\infty}(\mathbb{R})$, with continuous immersion, and the condition at infinity here
is that (a good representative of) $u\in BV(\mathbb{R})$ satisfies $\displaystyle{\lim_{t\to- \infty} u(t)=0}$.
If $N\geq 2$, we require that $u\in BV(\mathbb{R}^N)$ vanishes at infinity in a weak sense. We note that
$BV(\mathbb{R}^N)\subset L^{N/(N-1)}(\mathbb{R}^N)$, with continuous immersion, and $u$ belonging to
$L^{N/(N-1)}(\mathbb{R}^N)$ already guarantees that $u$ vanishes at infinity in a weak sense. Finally, we endow the homogeneous
$BV(\mathbb{R}^N)$ space with the norm given by the total variation of $Du$, namely
$$\|u\|_{BV(\mathbb{R}^N)}=|Du|(\mathbb{R}^N)=TV(u).$$
We refer to Section~1.12 in Meyer's book \cite{YMey} for further details.

About bounded domains, let us consider $\Omega\subset\R^N$, $N\geq 1$, to be a bounded domain with Lipschitz boundary.
$BV(\Omega)$ is the set of $L^1(\Omega)$ functions such that $Du$ is a bounded vector valued Radon measure on $\Omega$. Its norm is given by
$$\|u\|_{BV(\Omega)}=\|u\|_{L^1(\Omega)}+|u|_{BV(\Omega)}$$
where the seminorm $|u|_{BV(\Omega)}$ is just the total variation of $u$, that is, $|Du|(\Omega)=TV(u)$.
We recall that $BV(\Omega)$ is continuously contained in $L^{N/(N-1)}(\Omega)$ and that it is compactly contained in $L^1(\Omega)$.

Let
$$L^1_{\ast}(\Omega)=\left\{u\in L^1(\Omega):\ \int_{\Omega}u=0\right\}\text{ and }
BV_{\ast}(\Omega)=\left\{u\in BV(\Omega):\ \int_{\Omega}u=0\right\}.$$
Then we note that
$$\|u\|_{BV_{\ast}(\Omega)}=|u|_{BV(\Omega)}$$
is a norm on $BV_{\ast}(\Omega)$ which is topologically equivalent to the usual $BV$-norm.

Overall, we distinguish four different cases, in all of them $f\in H$.
\begin{enumerate}[i)]
\item\label{caso1} $F=BV(\R^N)$, $N\geq 1$, $H=L^2(\R^N)$, $R(u)=\|u\|_{BV(\R^N)}=|Du|(\R^N)$.
\item\label{caso2} Let $\Omega$ be a bounded domain of $\R^N$, $N\geq 1$, with Lipschitz boundary. Then $F=BV(\Omega)$,
$H=L^2(\Omega)$, $R(u)=\|u\|_{BV(\Omega)}$.
\item\label{caso3} Let $\Omega$ be a bounded domain of $\R^N$, $N\geq 1$, with Lipschitz boundary. Then $F=BV(\Omega)$,
$H=L^2(\Omega)$, $R(u)=|u|_{BV(\Omega)}=|Du|(\Omega)$.
\item\label{caso4} Let $\Omega$ be a bounded domain of $\R^N$, $N\geq 1$, with Lipschitz boundary. Then $F=BV_{\ast}(\Omega)$, 
$H=L^2_{\ast}(\Omega)$, $R(u)=\|u\|_{BV_{\ast}(\Omega)}=|u|_{BV(\Omega)}=|Du|(\Omega)$.
\end{enumerate}

In any of these cases, \eqref{structcondition} and \eqref{densecondition} hold. Morevorer,
existence, and uniqueness, of a solution to \eqref{minpbm} is guaranteed, hence Assumption~\ref{existenceassum} also holds.
Therefore, Theorem~\ref{1thm} is true and all the results of Section~\ref{Parssec} applies. In particular,  see for more details Theorem 3.2, Theorem 3.3 and Remark 3.4 in \cite{M-N-R}, in all these cases $v_n\to 0$ in $H$, that is, we can apply Proposition~\ref{multprop} with $\delta_0=0$ and $\tilde{v}=0$ and the multiscale decomposition and the Parseval-like identity of \eqref{multiscale+pars} hold true. We note that in \cite{M-N-R} the Parseval-like identity was stated only for $N=2$ but actually holds in any dimension $N\geq 1$.

\begin{oss}\label{normvsseminorm} We compare cases \ref{caso3}) and \ref{caso4}).
Let $f\in L^2(\Omega)$ and let us consider case \ref{caso3}), that is,
\begin{equation}\label{ROF}
\min\left\{\lambda_0\|f-u\|_{L^2(\Omega)}^2+|u|_{BV(\Omega)}:\ u\in BV(\Omega)\cap L^2(\Omega)  \right\}.
\end{equation}
Then the minimizer $u_0$ satisfies
$$\frac1{|\Omega|}\int_\Omega u_0=\frac1{|\Omega|}\int_\Omega f.$$

Hence if $f\in  L^2_{\ast}(\Omega)$, the two minimization problems, \eqref{ROF} and the one of case \ref{caso4}), namely
\begin{equation}\label{ROF2}
\min\left\{\lambda_0\|f-u\|_{L^2(\Omega)}^2+\|u\|_{BV_{\ast}(\Omega)}: u\in BV_{\ast}(\Omega)\cap L^2(\Omega)  \right\},
\end{equation}
are perfectly equivalent.

Moreover, for any $f\in L^2(\Omega)$, the minimizer $u_0$ of \eqref{ROF} is such that
$$\displaystyle{u_0=\frac1{|\Omega|}\int_{\Omega} f+\tilde{u}_0}$$
where $\tilde{u}_0\in BV_{\ast}(\Omega)$ solves \eqref{ROF2} with $f$ replaced by $\tilde{f}=\displaystyle{f-\frac1{|\Omega|}\int_{\Omega}f}  \in L^2_{\ast}(\Omega)$. Finally,
$v_0=f_0-u_0=\tilde{f}_0-\tilde{u}_0\in L^2_{\ast}(\Omega)$. We conclude that, up to removing the constant given by the mean of $f$ on $\Omega$, case \ref{caso3}) can be reduced to case \ref{caso4}), where the main advantage is that the penalization term is actually a norm. Moreover, in the multiscale procedure such an issue appears in the first step only, since $v_0\in L^2_{\ast}(\Omega)$, thus $v_n\in L^2_{\ast}(\Omega)$ for any $n\in\mathbb{N}$.
\end{oss}

Finally, in all cases the results of Subsection~\ref{mincarsec} are valid. In fact, for cases \ref{caso1}), \ref{caso2}) and \ref{caso4}), $R$ is indeed the norm of a Banach space. About case \ref{caso3}), this can be reduced to case \ref{caso4}) by Remark~\ref{normvsseminorm}.

\section{Regularization by Hilbert norms}\label{Hilbertregnormsec}
We consider the convex case. Namely,
we assume that $E$ is convex and that $E\ni\sigma\to d(\tilde{\Lambda},\Lambda(\sigma))$ is a continuous convex function.
Let $R:X\to [0,+\infty]$ be a convex function on $X$ satisfying the assumptions of Theorem~\ref{1thm}.
We assume that $\alpha\geq 1$ and $\beta\geq 1$.

We further assume that $R$ satisfies Assumption~\ref{strictconvexity}.
Let Assumption~\ref{stronger} be satisfied and let $\hat{\sigma}\in E$ be a solution to \eqref{minimal00}.
Let us call
\begin{equation}\label{Sdef}
S=\left\{\tilde{\sigma}\in E:\ R(\tilde{\sigma})<+\infty\text{ and }\ d(\tilde{\Lambda},\Lambda(\tilde{\sigma}))=\delta_0=\min_{\sigma\in X}d(\tilde{\Lambda},\Lambda(\sigma))\right\}.
\end{equation}
Note that $S$ is a convex set and, by Assumption~\ref{strictconvexity} on $R$, $\hat{\sigma}$ solving \eqref{minimal00} is unique,
that is, $S_{opt}=\{\hat{\sigma}\}$.

We call $\sigma_{-1}=0$. 
Clearly $\sigma_0=u_0\in E$ solution to \eqref{regularizedpbm} satisfies
 $R(u_0)\leq R(\hat{\sigma})\leq R(\tilde{\sigma})$ for any $\tilde{\sigma}\in S$.
 If $\sigma_0=u_0=\hat{\sigma}$, then $\sigma_n=\hat{\sigma}$ for any $n\geq 1$. For any $n\geq 0$, if $\sigma_{n-1}\in S$,
 then $\sigma_m=\sigma_{n-1}$ for any $m\geq n$. Let $n\geq 0$ and
let $u_n$ be the solution to \eqref{inductiveconstr} and $\sigma_n=\sigma_{n-1}+u_n$.
We consider the segment connecting $\sigma_n$ to $\tilde{\sigma}\in S$, that is,
$$l=\{x(t)=\sigma_n+t(\tilde{\sigma}-\sigma_n),\ t\in [0,1]\}$$
and we claim that
\begin{equation}\label{exteriorsegment}
R(x-\sigma_{n-1})>R(u_n)\quad\text{for any }x\in l\backslash\{\sigma_{n}\}.
\end{equation}
Let us assume that $l\backslash\{\sigma_{n}\}$ is not empty.
Then, by contradiction, assume that $R(x(t_0)-\sigma_{n-1})\leq R(u_n)$ for some $0<t_0\leq 1$.
By convexity, we have
\begin{equation}\label{hereconvexityused}
\big(d(\tilde{\Lambda},\Lambda(x(t))\big)^{\alpha}\leq \big(d(\tilde{\Lambda},\Lambda(\sigma_n)\big)^{\alpha}\quad\text{for any }0<t\leq 1
\end{equation}
and by Assumption~\ref{strictconvexity} on $R$
$$(R(x(t)-\sigma_{n-1}))^{\beta}< (R(u_n))^{\beta}\quad\text{for any }0<t<t_0,$$
thus contradicting the minimality of $u_n$.

\begin{oss}\label{convexneededoss}
To prove \eqref{exteriorsegment}, we have used the convexity of $E\ni\sigma\to d(\tilde{\Lambda},\Lambda(\sigma))$ only to show \eqref{hereconvexityused}. Indeed, convexity implies that, being $\tilde{\sigma}$ a minimizer, the function $[0,1]\ni t\mapsto d(\tilde{\Lambda},\Lambda(x(t))$ is decreasing with respect to $t$.
\end{oss}

\subsection{Convergence results}\label{convergressec}
We show that the simple observation contained in \eqref{exteriorsegment} has a very important consequence when the regularization $R$ is chosen to be the norm of a Hilbert space,
that is, when Assumption~\ref{hilbertassum} is satisfied.

In fact, under Assumption~\ref{hilbertassum} and assuming that
\eqref{compactnesscondition} and \eqref{lsccondition} hold, we have for any $\tilde{\sigma}\in S$
\begin{equation}\label{decreasing}
\|\tilde{\sigma}-\sigma_{n}\|_{H_1}\leq \|\tilde{\sigma}-\sigma_{n-1}\|_{H_1}\quad\text{for any }n\geq 0.
\end{equation}
Property \eqref{decreasing} follows from the property \eqref{exteriorsegment}. In fact, assuming that $\sigma_n$ is different from $\tilde{\sigma}$,
for any $0<t\leq 1$
$$\|x(t)-\sigma_{n-1}\|_{H_1}^2=\|\sigma_n+t(\tilde{\sigma}-\sigma_n)-\sigma_{n-1}\|_{H_1}^2>\|\sigma_n-\sigma_{n-1}\|_{H_1}^2,$$
that is,
$$t^2\|\tilde{\sigma}-\sigma_n\|_{H_1}^2+2t\langle \sigma_n-\sigma_{n-1},\tilde{\sigma}-\sigma_n\rangle_{H_1}>0\quad\text{for any }0<t\leq 1,$$
which implies
\begin{equation}\label{eq1}
\langle \sigma_n-\sigma_{n-1},\tilde{\sigma}-\sigma_n\rangle_{H_1}\geq 0.
\end{equation}
We conclude that
\begin{multline}\label{eq2}
\|\tilde{\sigma}-\sigma_{n-1}\|_{H_1}^2\\=
\|\tilde{\sigma}-\sigma_n\|_{H_1}^2+\|\sigma_n-\sigma_{n-1}\|_{H_1}^2+2\langle \sigma_n-\sigma_{n-1},\tilde{\sigma}-\sigma_n\rangle_{H_1}\geq \|\tilde{\sigma}-\sigma_n\|_{H_1}^2
\end{multline}
and \eqref{decreasing} is proved.

\begin{teo}\label{hilbertregthm}
Let $E\subset X$ be convex and let $E\ni\sigma\to d(\tilde{\Lambda},\Lambda(\sigma))$ be a continuous convex function.
Let us assume that $R$ is as in Assumption~\ref{hilbertassum} and that it satisfies
\eqref{compactnesscondition} and \eqref{lsccondition}. Let 
Assumption~\ref{stronger} be satisfied. Let $\alpha\geq 1$ and $\beta\geq 1$.

Then, calling $\sigma_{-1}=0$, for any $n\geq 0$ we have
\begin{equation}\label{result1}
\mathrm{dist}_{H_1}(\sigma_n,S)\leq \mathrm{dist}_{H_1}(\sigma_{n-1},S)\leq \|\hat{\sigma}\|_{H_1}.
\end{equation}
Furthermore,
 if $\displaystyle \lim_n\lambda_n=+\infty$, we have
\begin{equation}\label{result2}
\lim_n d(\tilde{\Lambda},\Lambda(\sigma_n))=\delta_0=\min_{\sigma\in X}d(\tilde{\Lambda},\Lambda(\sigma)),
\end{equation}
and 
\begin{equation}\label{result3}
\lim_n \mathrm{dist}_X(\sigma_n,S)=0.
\end{equation}
Finally, for any subsequence
$\{\sigma_{n_k}\}_{k\in\mathbb{N}}$
there exists a subsequence $\{\sigma_{n_{k_j}}\}_{j\in\mathbb{N}}$ converging to $\sigma_{\infty}\in S$ as $j\to+\infty$.
\end{teo}

\begin{oss}
Contrary to the classical multiscale procedure, no control on the order of infinity of the sequence $\lambda_n$ is required. This is essentially due to Assumption~\ref{hilbertassum}.
Moreover, by Assumption~\ref{stronger}, we do not require condition \eqref{densecondition}, that is, that $H_1\cap E$ is dense in $E$.
\end{oss}

\begin{proof} The decreasing property of \eqref{result1} is an immediate consequence of \eqref{decreasing}.

About \eqref{result2}, we follow the proof of \cite[Theorem~2.1]{M-N-R}.
Actually the argument here is much simpler. By contradiction, assume that $\delta_0<\varepsilon_0$.
We have that
$$\delta_0=d(\tilde{\Lambda},\Lambda(\hat{\sigma}))<\varepsilon_0\quad\text{and}\quad\|\hat{\sigma}\|_{H_1}<+\infty.$$
For any $n\geq 0$,
$$
\lambda_n\big(d(\tilde{\Lambda},\Lambda(\sigma_n))\big)^{\alpha}+\|u_n\|_{H_1}^{\beta}\leq
\lambda_n\big(d(\tilde{\Lambda},\Lambda(\hat{\sigma}))\big)^{\alpha}+\|\hat{\sigma}-\sigma_{n-1}\|_{H_1}^{\beta}
$$
hence
$$0<\lambda_n\big(\varepsilon_0^{\alpha}-\delta_0^{\alpha}\big)\leq \|\hat{\sigma}-\sigma_{n-1}\|_{H_1}^{\beta}\leq \|\hat{\sigma}\|_{H_1}^{\beta}.$$
If $\displaystyle \lim_n\lambda_n=+\infty$, we obtain a contradiction and \eqref{result2} is proved.

Before showing \eqref{result3}, we prove the last property. Since $\|\sigma_n\|_{H_1}\leq 2\|\hat{\sigma}\|_{H_1}$ for any $n\geq 0$, by compactness, there exists a subsequence of
$\{\sigma_{n_{k_j}}\}_{j\in\mathbb{N}}$ converging in $X$ to $\sigma_{\infty}$. By \eqref{result2}, we immediately conclude that
$d(\tilde{\Lambda},\Lambda(\sigma_{\infty}))=\delta_0$ and, by lower semicontinuity, $\|\sigma_{\infty}\|_{H_1}\leq 2\|\hat{\sigma}\|_{H_1}$, that is, $\sigma_{\infty}\in S$.
Thus the last property holds true.

About \eqref{result3}, by contradiction let us assume that it is not true. Then there exists $\varepsilon>0$ and a subsequence
$\{\sigma_{n_k}\}_{k\in\mathbb{N}}$ such that $\mathrm{dist}_X(\sigma_{n_k},S)> \varepsilon$ for any $k\in \mathbb{N}$. But this contradicts the last property.\end{proof}

\begin{oss}As it can be seen from the proof and using Remark~\ref{convexneededoss}, in Theorem~\ref{hilbertregthm} we can replace the assumption that $E\ni\sigma\to d(\tilde{\Lambda},\Lambda(\sigma))$ is convex with the following.
 It is enough to assume that for $\hat{\sigma}$ as in Assumption~\ref{stronger} and for any $\sigma\in E$ we have
$$[0,1]\ni t\mapsto d(\tilde{\Lambda},\Lambda(\hat{\sigma}+t(\sigma-\hat{\sigma})))\ \text{is increasing with respect to }t.$$
The only difference in the result is that we need to replace \eqref{result1} with
\begin{equation}\label{result1new}
\|\sigma_n-\hat{\sigma}\|_{H_1}\leq \|\sigma_{n-1}-\hat{\sigma}\|_{H_1}\leq \|\hat{\sigma}\|_{H_1}\quad\text{for any }n\geq 0.
\end{equation}
\end{oss}

\begin{cor}\label{linearcase} Let $E=X$, $Y$ be a strictly convex Banach space and $\Lambda:X\to Y$ be a bounded linear operator.
Let us assume that $R$ is as in Assumption~\ref{hilbertassum}, it satisfies
\eqref{compactnesscondition} and \eqref{lsccondition}, and that $H_1\cap X$ is a closed subspace of $H_1$. In other words, we assume that $H_1\subset X$.
Let Assumption~\ref{stronger} be satisfied and let $\alpha\geq 1$ and $\beta\geq 1$.

Then, calling $\sigma_{-1}=0$, for any $n\geq 0$ we have
\begin{equation}\label{result1bis}
\mathrm{dist}_{H_1}(\sigma_n,S)\leq \mathrm{dist}_{H_1}(\sigma_{n-1},S)\leq \|\hat{\sigma}\|_{H_1}.
\end{equation}
Furthermore,
 if $\displaystyle \lim_n\lambda_n=+\infty$, we have
\begin{equation}\label{result2bis}
\lim_n d(\tilde{\Lambda},\Lambda(\sigma_n))=\delta_0=\min_{\sigma\in X}d(\tilde{\Lambda},\Lambda(\sigma)),
\end{equation}
and 
\begin{equation}\label{result4}
\lim_n \sigma_n=\hat{\sigma}\quad\text{in }X,
\end{equation}
that is,
$$\hat{\sigma}=\sum_{n=0}^{+\infty}u_n\quad\text{in }X.$$
\end{cor}

\begin{proof} Clearly \eqref{result1bis} and \eqref{result2bis} hold tue. It remains to prove \eqref{result4}.
First of all, we note that, since $H_1\subset X$ and the immersion is compact, if $v_n\rightharpoonup v$ weakly in $H_1$, then $v_n\to v$ strongly in $X$.
 We investigate the structure of $S$. 
We note that $S$ is a convex set which is closed in $H_1$. 
Moreover, let $\tilde{\sigma}\in S$, with $\tilde{\sigma}\neq \hat{\sigma}$. Then, $\|\tilde{\Lambda}-\Lambda(\sigma)\|_Y=\|\tilde{\Lambda}-\Lambda(\hat{\sigma})\|_Y$ implies, by the strict convexity of $Y$, that $\Lambda(\tilde{\sigma})=\Lambda(\hat{\sigma})$. We conclude that
$$S=\hat{\sigma}+\big(\mathrm{Ker}(\Lambda)\cap H_1\big)=\hat{\sigma}+K$$
where $K$ is a suitable closed subspace of $H_1$. Note that $\|\hat{\sigma}\|\leq \|\tilde{\sigma}\|$ for any
$\tilde{\sigma}\in S$, hence $\hat{\sigma}$ is orthogonal to $K$, since $\hat{\sigma}$ is the projection of $0$ on $S$.

When $K=\{0\}$, that is, when $\Lambda$ is injective on $H_1$, $S=\{\hat{\sigma}\}$ and \eqref{result4} immediately follows from \eqref{result3}. Otherwise, assume that $K$ is not trivial. We claim that, for any $n\geq 0$, $u_n$, hence $\sigma_n$, is orthogonal to $K$ in $H_1$. Since $\sigma_{-1}=0$, it is enough to show by induction that $\sigma_{n-1}$ orthogonal to $K$ implies that $u_n$ is orthogonal to $K$. Let $v\in K\backslash \{0\}$ and let $\tilde{\sigma}=\hat{\sigma}+\lambda v$ for some
constant $\lambda\in\mathbb{R}$. Then, recalling \eqref{eq1}, we have
$$\langle u_n,\lambda v+\hat{\sigma}-\sigma_n \rangle_{H_1}\geq 0\quad \text{for any }\lambda\in \mathbb{R}$$
which implies
$$\langle u_n, v\rangle_{H_1}= 0\quad
\text{and}\quad\langle u_n,\hat{\sigma}-\sigma_n \rangle_{H_1}\geq 0$$
and the required orthogonality is proved.

By contradiction, let us assume that for some subsequence $\{\sigma_{n_k}\}_{k\in\mathbb{N}}$ we have
$\sigma_{n_k}\to \sigma_{\infty}\in S$ in $X$ with $\sigma_{\infty}\neq \hat{\sigma}$, that is, $\sigma_{\infty}=\hat{\sigma}+v$ for some $v\in K\backslash\{0\}$. But, up to a subsequence which we do not relabel,
$\sigma_{n_k}\rightharpoonup \sigma_{\infty}$ weakly in $H_1$. We deduce that $$0=\langle \sigma_{n_k},\tilde{v}\rangle_{H_1}\to
\langle \sigma_{\infty},\tilde{v}\rangle_{H_1}\quad\text{for any }\tilde{v}\in K,$$
hence $\sigma_{\infty}=\hat{\sigma}+v$ is orthogonal to $K$. Since $\hat{\sigma}$ is orthogonal to $K$, we obtain a contradiction if $v\neq 0$.\end{proof}

\subsection{Improvement of the stability of the reconstruction}\label{stabsec}
We show how the stability of the reconstruction for a linear inverse ill-posed problem may be improved if we use the multiscale procedure with a Hilbert norm as regularization.
Let us consider the assumptions of Corollary~\ref{linearcase}.
Let $Y=H$ be a Hilbert space with scalar product $\langle\cdot,\cdot\rangle_H$ and norm $\|\cdot\|_H$. Let $X$ be a Banach space, $E=X$, and let $\Lambda:X\to H$ be a bounded linear operator.
Let us assume that $R$ is as in Assumption~\ref{hilbertassum}, it satisfies
\eqref{compactnesscondition} and \eqref{lsccondition}, and that $H_1\cap X$ is a closed subspace of $H_1$. In other words, we assume that $H_1\subset X$. We consider $\alpha=2$ and $\beta=1$.
Let $\hat{\sigma}\in H_1$ be such that $\tilde{\Lambda}=\Lambda(\hat{\sigma})$, hence
Assumption~\ref{stronger} is satisfied. Moreover, we assume that $\Lambda$ is injective on $H_1$, thus $\hat{\sigma}$ is the only solution to the inverse problem
$$\Lambda(\sigma)=\tilde{\Lambda}$$
under the a priori hypothesis that $\sigma\in H_1$. We note that we are considering exact data without any noise.

Let $\{\lambda_n\}_{n\geq 0}$ be a strictly increasing sequence of positive parameters such that $\displaystyle \lim_n\lambda_n=+\infty$. By Corollary~\ref{linearcase},
we know that the sequence $\sigma_n$ constructed by our multiscale procedure is converging to $\hat{\sigma}$.
We now consider a corresponding single-step regularization. Namely, let
 $\tilde{\sigma}_n$, for any $n\geq 0$, be the solution to
\begin{equation}\label{eq:single_step_procedure}
\min
\left\{\lambda_n\|\tilde{\Lambda}-\Lambda(\sigma)\|_H^2+\|\sigma\|_{H_1}:\ \sigma\in X\right\}.
\end{equation}
We note that $\sigma_n$ and $\tilde{\sigma}_n$ are uniquely identified.
Moreover, we have that for any $n\geq 0$
$$\|\hat{\sigma}-\tilde{\sigma}_n\|_{H_1}\leq \|\hat{\sigma}\|_{H_1}$$
and
$$\|\Lambda(\hat{\sigma}-\tilde{\sigma}_n)\|^2_H\leq \frac{\|\hat{\sigma}\|_{H_1}}{\lambda_n}.$$
We can conclude that also $\tilde{\sigma}_n$ converges to $\hat{\sigma}$ in $X$ as $n\to+\infty$.
On the other hand, we recall that for any $n\geq 0$
\begin{equation}\label{decreasingH_1norm}
\|\hat{\sigma}-\sigma_n\|_{H_1}\leq \|\hat{\sigma}-\sigma_{n-1}\|_{H_1}\leq \|\hat{\sigma}\|_{H_1}
\end{equation}
and
$$\|\Lambda(\hat{\sigma}-\sigma_n)\|^2_H\leq \frac{\|\hat{\sigma}-\sigma_n\|_{H_1}}{\lambda_n}$$
since $\|\hat{\sigma}-\sigma_{n-1}\|_{H_1}\leq \|\hat{\sigma}-\sigma_n\|_{H_1}+\|u_n\|_{H_1}$.

Let us assume that
$$\|\Lambda(\hat{\sigma})\circ\Lambda\|_{\ast}>\frac{1}{2\lambda_0}.$$
Then, for any $n\geq 0$, 
\begin{equation}\label{scale-defin}
\|\Lambda(\hat{\sigma}-\sigma_n)\circ\Lambda\|_{\ast}=\|\Lambda(\hat{\sigma}-\tilde{\sigma}_n)\circ\Lambda\|_{\ast}=\frac{1}{2\lambda_n}.
\end{equation}
We can consider this, roughly speaking, as the resolution or the scale we wish to obtain by our reconstruction method for the inverse problem. Hence the two methods give the same resolution.
However, the multiscale method behaves better with respect to stability.
In fact, let us assume that the following conditional stability estimate holds.
For any $C>0$, there exists a modulus of continuity $\omega_C$ such that if $\|\sigma\|_{H_1}\leq C$, then
$$\|\sigma\|_{X}\leq \omega_{C}(\|\Lambda(\sigma)\|_H).$$
Clearly, $\omega_{C}\leq \omega_{C_1}$ for any $0<C\leq C_1$.
Then, unless we have any further information, we can just compare
\begin{equation}\label{eq:estimate1}
\|\hat{\sigma}-\tilde{\sigma}_n\|_X\leq \omega_{\|\hat{\sigma}\|_{H_1}}\left(\sqrt\frac{\|\hat{\sigma}\|_{H_1}}{\lambda_n}\right)
\end{equation}
with 
\begin{equation}\label{eq:estimate2}
\|\hat{\sigma}-\sigma_n\|_X\leq \omega_{\|\hat{\sigma}-\sigma_n\|_{H_1}}\left(\sqrt\frac{\|\hat{\sigma}-\sigma_n\|_{H_1}}{\lambda_n}\right).
\end{equation}
Recalling \eqref{decreasingH_1norm},
it is evident that the second one is better and that it can potentially improve with $n$, both in the error, besides the factor $\lambda_n$, and in the a priori bound.

\subsection{Image deblurring with Hilbert regularization: numerical illustration}\label{numerics}
In the following, we compare numerically the multiscale decomposition method and the single-step procedure defined by \eqref{eq:single_step_procedure} on an image deblurring test problem. Our goal is to verify that the multiscale approach practically improves the stability of the image reconstruction, as theoretically suggested by the estimates \eqref{eq:estimate1}-\eqref{eq:estimate2} given in the previous section. All the numerical experiments hereby reported have been performed by means of routines implemented in Matlab R2023a.

Let us assume that the unknown image belongs to the space $H^1(\Omega)$, $\Omega=(0,1)^2$, equipped with the usual norm $\|\sigma\|_{H^1(\Omega)}=(\|\sigma\|^2_{L^2(\Omega)}+\|\nabla \sigma\|^2_{L^2(\Omega)})^{\frac{1}{2}}$. Note that this regularization space preserves smoothness but does not allow for edges \cite{Chan-et-al-2006,Mang-et-al-2016}. Furthermore, we let $\Lambda:L^2(\Omega)\rightarrow L^2(\Omega)$ be a continuous linear operator modelling the blurring process on the image. Given an observed image $\tilde{\Lambda}\in L^2(\Omega)$, we aim at solving the linear inverse problem $\Lambda(\hat{\sigma})=\tilde{\Lambda}$, where $\hat{\sigma}\in H^1(\Omega)$ denotes the clean image that needs be restored. Under the assumption that a unique solution $\hat{\sigma}$ to the inverse problem exists, we can conclude that the sequence $\{\sigma_n\}_{n\geq 0}$ generated by the multiscale procedure, as well as the sequence $\{\tilde{\sigma}_n\}_{n\geq 0}$ generated by the single-step procedure, converge to $\hat{\sigma}$ in $L^2(\Omega)$ as $n\rightarrow +\infty$.

We discretize the problem as follows. Let $h=1/N$ be the discretization step and $\hat{\sigma}_h=(\hat{\sigma}((i-1)h,(j-1)h))_{i,j=1}^N$ the (vectorized) matrix in $\R^{N^2}$ representing a discretized version of the clean image $\hat{\sigma}$ in $H^1(\Omega)$. Analogously, $\Lambda_h\in \R^{N^2\times N^2}$ denotes the blurring matrix obtained by discretization of the blurring operator $\Lambda$, and $\tilde{\Lambda}_h\in \R^{N^2}$ represents the discretization of the observed image $\tilde{\Lambda}$. Then, our discrete problem consists in finding $\hat{\sigma}_h$ such that $\Lambda_h\hat{\sigma}_h=\tilde{\Lambda}_h$.

For a discrete image $u=(u_{i,j})_{i,j=1}^N$ in $\R^{N^2}$, we denote the $0$-th and $1$-st order Tikhonov terms as
$$
\|u\|_2^2=\sum_{i,j}^N u_{i,j}^2, \quad \|\nabla u\|_2^2 = \sum_{i,j=1}^{N}(u_{i+1,j}-u_{i,j})^2+(u_{i,j+1}-u_{i,j})^2.
$$     
Given a parameter $\delta>0$, we consider the following discretized and smoothed version of the $H^1$-norm:
\begin{equation*}
R_{h,H^1}(u)=\left(\|u\|_2^2+\|\nabla u\|_2^2+\delta^2\right)^{\frac{1}{2}}.
\end{equation*}
Let $\{\lambda_n\}_{n\in\N}$ be a sequence of positive parameters with $\displaystyle \lim_{n}\lambda_n=+\infty$. Then, the multiscale procedure generates the sequence of iterates $\{{\sigma}_{h,n}\}_{n\in\N}$ in $\R^{N^2}$ as follows
\begin{eqnarray}\label{eq:multiscale-1}
&\sigma_{h,1} =u_{h,1}= \underset{u\in\R^{N^2}}{\operatorname{argmin}} \ \lambda_1\|\tilde{\Lambda}_h-\Lambda_hu\|_2^2 + R_{h,H^1}(u)\\
%\end{equation}
%\begin{equation}
\label{eq:multiscale-2}
&\begin{cases}
u_{h,n} = \underset{u\in\R^{N^2}}{\operatorname{argmin}} \  \lambda_n\|\tilde{\Lambda}_h-\Lambda_h (u+\sigma_{h,n-1})\|_2^2  + R_{h,H^1}(u)\\
\sigma_{h,n} = \sigma_{h,n-1}+u_{h,n}, \quad n=2,3,\ldots
\end{cases}
\end{eqnarray}
Likewise, the sequence $\{\tilde{\sigma}_{h,n}\}_{n\in \N}$ generated by the single-step procedure is given by
\begin{equation}\label{eq:singlestep}
\tilde{\sigma}_{h,n} = \underset{\sigma\in\R^{N^2}}{\operatorname{argmin}} \  \lambda_n \|\tilde{\Lambda}_h-\Lambda_h \sigma\|_2^2 + R_{h,H^1}(\sigma), \quad n=1,2,\ldots
\end{equation}
\begin{center}
\begin{figure}[t!]
\begin{tabular}{cc}
\includegraphics[scale = 0.6]{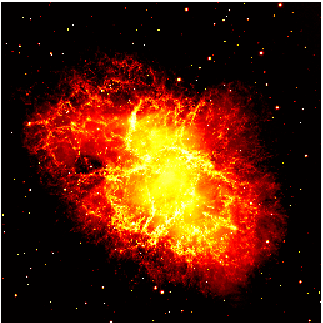}&\includegraphics[scale = 0.6]{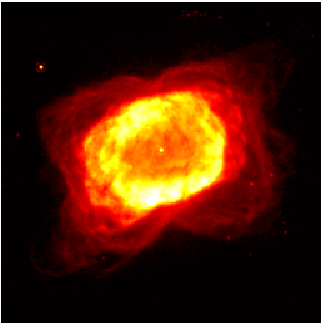}  
\end{tabular}
\caption{Test images. Crab Nebula NGC 1952 (left) and planetary Nebula NGC 7027 (right).}
\label{fig:0}
\end{figure}
\end{center}

We consider two Hubble Space Telescope $256\times 256$ grayscale images, representing the Crab Nebula NGC 1952 and the planetary Nebula NGC 7027, respectively, see Figure \ref{fig:0}. These images have been previously employed in several astronomical image deconvolution tests \cite{Bertero-et-al-2018,Bonettini-et-al-2018,Prato-et-al-2013,Stagliano-et-al-2011}. We note that the $H^1$ regularization is well-suited for these images, as they are smooth objects whose borders fade smoothly into the background.

Two different tests are set up. In the first test, we aim at restoring the artificially blurred images of the two nebulas, assuming that no noise is present in the data. We simulate numerically the blurred images by convolving each test image with a Gaussian kernel of size $9\times 9$ and standard deviation $2$. The top right corner of the resulting blurred images are reported in Figure \ref{fig:1} (first column). In the second test, we also corrupt the blurred images by adding white Gaussian noise with variance $0.01$. The top right corner of the resulting blurred and noisy images are shown in Figure \ref{fig:3} (first column).  

%For all tests, we choose as initial guesses for \eqref{eq:multiscale-1} and \eqref{eq:singlestep} the observed image $\tilde{\Lambda}_h$.
 In the noiseless tests, we set $\lambda_n = 2^{n-5}$, $n=1,\ldots,40$, as the sequence of regularization parameters in both the multiscale procedure \eqref{eq:multiscale-1}-\eqref{eq:multiscale-2} and the single-step procedure \eqref{eq:singlestep}; likewise, in the noisy tests, we choose $\lambda_n = 2^{n-5}$, $n=1,\ldots,20$, for both methods. Note that, at each iteration, both procedures require the solution of a minimization subproblem of the form $\min_{\sigma \in \R^{N^2}} \ F_{h,n}(\sigma)$, where $F_{h,n}:\R^{N^2}\rightarrow \R$ is continuously differentiable and convex. We solve these subproblems by means of a gradient method of the form
$$
\begin{cases}
\sigma_{h,n}^{(0)}\in\R^{N^2}\\
\sigma_{h,n}^{(\ell+1)} = \sigma_{h,n}^{(\ell)}-\alpha_{\ell}\lambda_{\ell}\nabla F_{h,n}(\sigma_{h,n}^{(\ell)}),\quad \ell=0,1,\ldots
\end{cases}
$$  
where the steplength $\alpha_{\ell}>0$ belongs to a compact subset of the positive real numbers set and is computed by adaptive alternation the two Barzilai-Borwein rules \cite{Bonettini-et-al-2009}, whereas the linesearch parameter $\lambda_{\ell}\in (0,1]$ is computed by performing an Armijo linesearch along the descent direction.
Note that such a scheme is guaranteed to converge to a minimizer of $F_{h,n}$ in the convex setting for any choice of the initial guess $\sigma_{h,n}^{(0)}$, see e.g. \cite{Bonettini-et-al-2015,Iusem-et-al-2003}. For the multiscale approach, the initial guess of the subproblem solver is chosen as either the observed image $\tilde{\Lambda}_h$ for $n=1$, or the image $\tilde{\Lambda}_h-\Lambda_h \sigma_{h,n-1}$ for each $n\geq 2$. For the single-step procedure, we let $\tilde{\Lambda}_h$ be the initial guess of the subproblem solver for all $n\in\N$. For both procedures, we stop the inner sequence $\{\sigma_{h,n}^{(\ell)}\}_{\ell\geq 0}$ when either a maximum number of $\ell_{\max} = 2\cdot 10^{4}$ iterations is reached or the following stopping criterion is met
\begin{equation*}
\frac{\|F_{h,n}(\sigma_{h,n}^{(\ell+1)})-F_{h,n}(\sigma_{h,n}^{(\ell)})\|}{\|F_{h,n}(\sigma_{h,n}^{(\ell+1)})\|}\leq \tau \quad \text{and} \quad \frac{\|\nabla F_{h,n}(\sigma_{h,n}^{(\ell+1)})\|}{\|\nabla F_{h,n}(\sigma_{h,n}^{(0)})\|}\leq \tau, \quad \tau = 10^{-4}.
\end{equation*}

\begin{center}
\begin{figure}[t!]
\begin{tabular}{ccc}
\includegraphics[scale = 0.6]{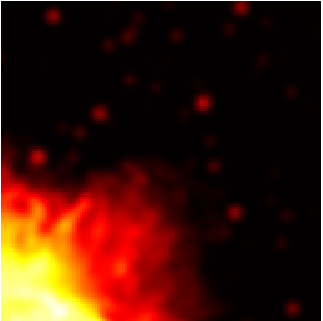}&\includegraphics[scale = 0.6]{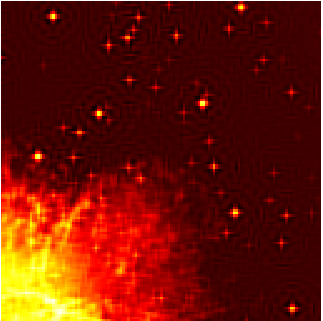} & \includegraphics[scale = 0.6]{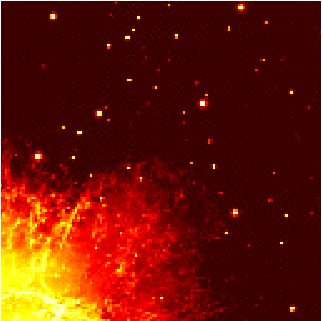} \\
\includegraphics[scale = 0.6]{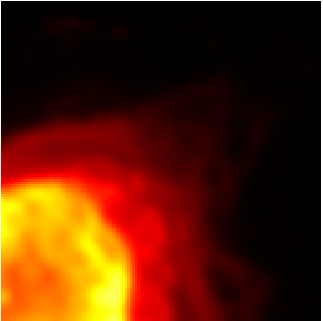}&\includegraphics[scale = 0.6]{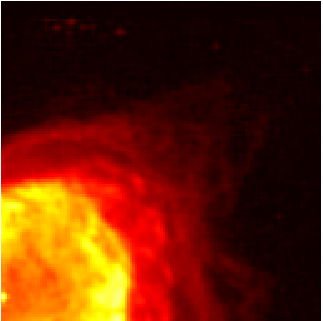} & \includegraphics[scale = 0.6]{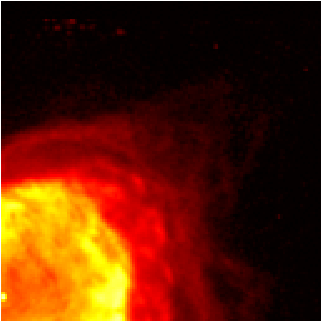} 
\end{tabular}
\caption{Noiseless tests. Zoomed details of the Crab Nebula NGC 1952 (first row) and the planetary Nebula NGC 7027 (second row). From left to right: blurred image, best deblurred image with single-step procedure, and best deblurred image with multiscale procedure.}
\label{fig:1}
\end{figure}
\end{center}

Focusing on the noiseless tests, we show in Figure \ref{fig:1} the top right corner of the deblurred images provided by the single-step regularization method (second column) and the multiscale method (third column), respectively, equipped with the values of the regularization parameter $\lambda_n$ that yield the lowest reconstruction error for both methods. For the Crab Nebula NGC 1952 (first row), the single-step regularization method provides its best deblurred image for $\lambda_{30}=2^{25}$ corresponding to a relative reconstruction error $\|\tilde{\sigma}_{h,n}-\hat{\sigma}_{h}\|/\|\hat{\sigma}_{h}\|$ of $0.140$, whereas the multiscale method provides its best deblurred image at the last iteration, i.e. for $\lambda_{40}=2^{35}$, with a relative reconstruction error $\|\sigma_{h,n}-\hat{\sigma}_{h}\|/\|\hat{\sigma}_{h}\|$ of $0.050$. For the planetary Nebula NGC 7027 (second row), both methods provide their best deblurred image at the last iteration, however the single-step regularization method yields an error of $0.016$, whereas the multiscale method retains an error of $0.004$. Therefore, for both tests images, the multiscale method yields the lowest reconstruction error and provides the best and most detailed image, as it can be seen by looking at the zoomed details in Figure \ref{fig:1}. 

Figure \ref{fig:2} shows the decrease of the relative reconstruction errors of the two methods in the noiseless tests. From these plots, we can see that the single-step regularization method becomes numerically unstable after the first $15-20$ iterations, whereas the multiscale method outperforms significantly the single-step method for most iterations. These results are coherent with the theoretical estimates \eqref{eq:estimate1}-\eqref{eq:estimate2}, which suggested the superiority of the multiscale method in terms of stability of the reconstruction.
\begin{center}
\begin{figure}[t!]
\begin{tabular}{ccc}
\includegraphics[scale = 0.6]{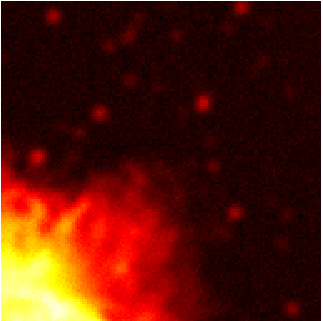}&\includegraphics[scale = 0.6]{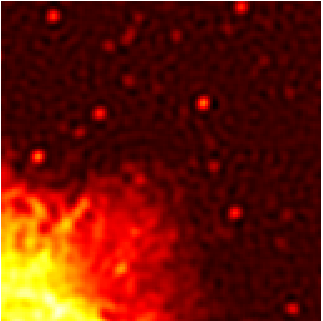} & \includegraphics[scale = 0.6]{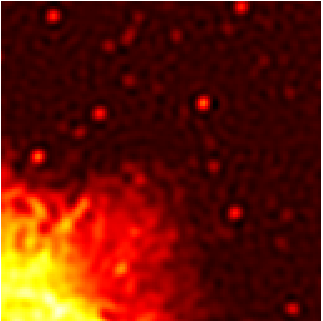} \\
\includegraphics[scale = 0.6]{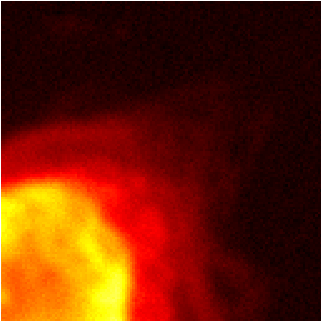}&\includegraphics[scale = 0.6]{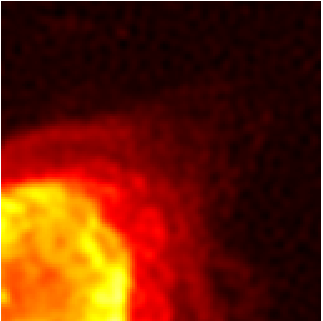} & \includegraphics[scale = 0.6]{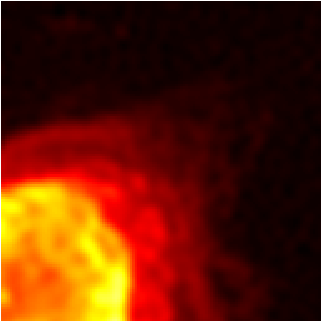} 
\end{tabular}
\caption{Noisy tests. Zoomed details of the Crab Nebula NGC 1952 (first row) and the planetary Nebula NGC 7027 (second row). From left to right: blurred image, best deblurred image with single-step procedure, and best deblurred image with multiscale procedure.}
\label{fig:3}
\end{figure}
\end{center}

Regarding the noisy tests, we report the top right corner of the deblurred images in Figure \ref{fig:3}. For the Crab Nebula NGC 1952, the two methods provides the same minimum value of the reconstruction error ($0.2$), although the multiscale method reaches the minimum three iterations later than the single-step regularization method, i.e. with $\lambda_{9}=2^4$. Analogously, for the planetary Nebula NGC 7027, the multiscale method reaches a slightly smaller value of the reconstruction error ($0.041$ against $0.047$) three iterations later than its competitor, with $\lambda_7 = 2^2$. Visually, the two deblurred images of the Crab Nebula NGC 1952 look comparable, whereas the multiscale method provides a less noisy reconstruction of the second test image. The plots of the relative reconstruction error in Figure \ref{fig:4} suggest that the multiscale method is more robust to the presence of noise, as the minimum values are reached at a later stage, which according
to \eqref{scale-defin} corresponds to a finer scale,
 and are kept for longer iterations than the single-step approach.

\begin{center}
\begin{figure}[t!]
\begin{tabular}{cc}
\includegraphics[scale = 0.4]{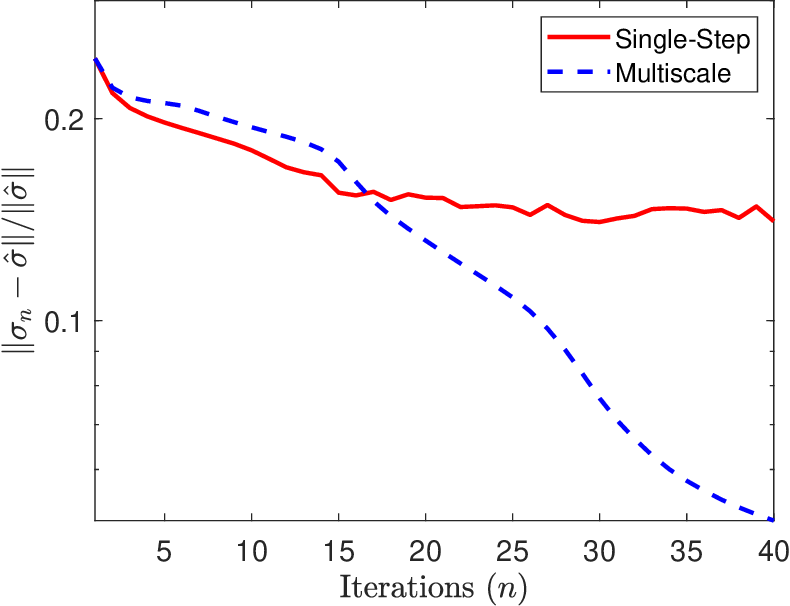}   \includegraphics[scale = 0.4]{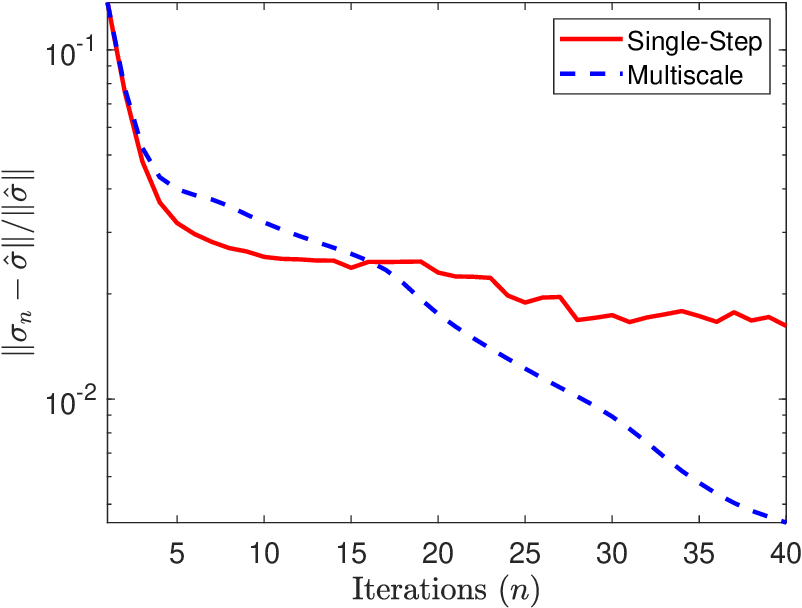}
\end{tabular}
\caption{Noiseless tests. Decrease of the relative error $\|\sigma_{h,n}-\hat{\sigma}_{h}\|/\|\hat{\sigma}_{h}\|$ for the multiscale method (blue, dashed line) and $\|\tilde{\sigma}_{h,n}-\hat{\sigma}_{h}\|/\|\hat{\sigma}_{h}\|$ for the single-step regularization method (red, solid line), with respect to the iteration number. Left: Nebula NGC 1952. Right: Nebula NGC 7027.}
\label{fig:2}
\end{figure}
\end{center}

\begin{center}
\begin{figure}[t!]
\begin{tabular}{cc}
\includegraphics[scale = 0.4]{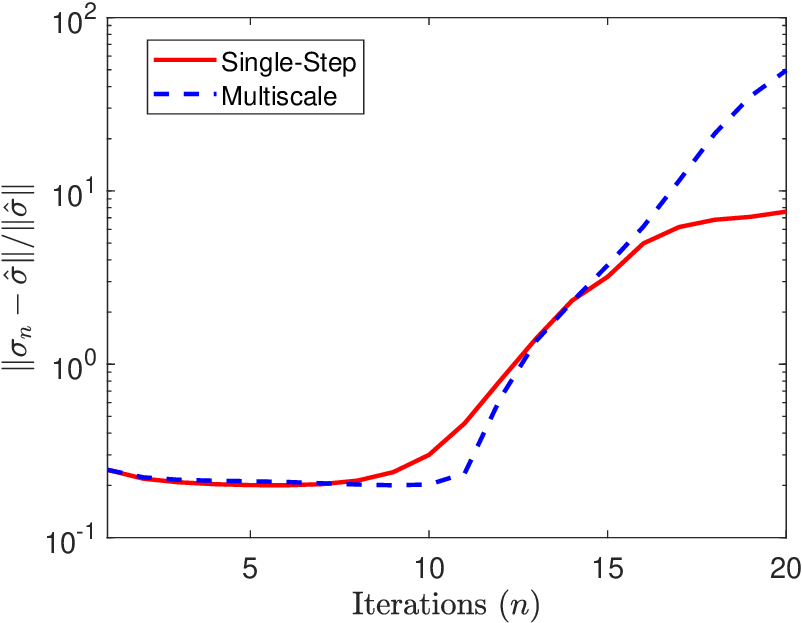}   \includegraphics[scale = 0.4]{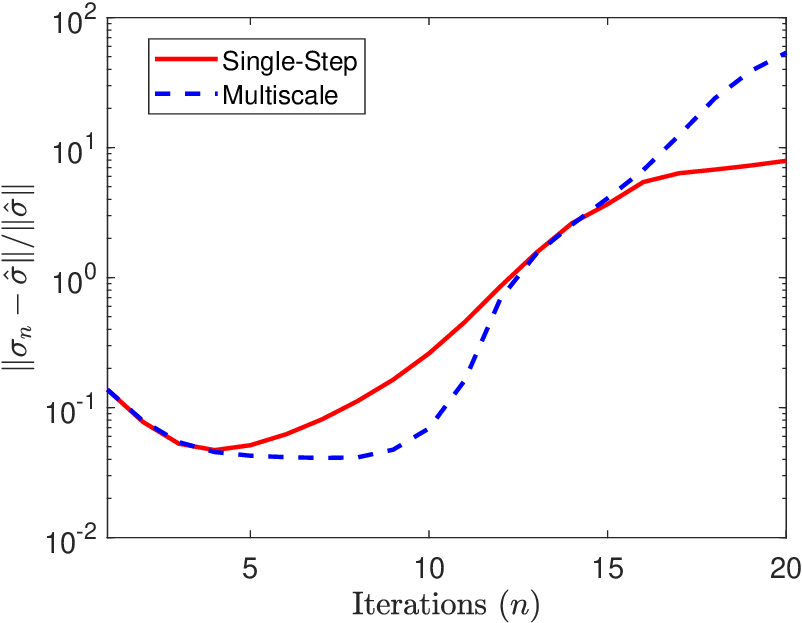}
\end{tabular}
\caption{Noisy tests. Decrease of the relative error $\|\sigma_{h,n}-\hat{\sigma}_{h}\|/\|\hat{\sigma}_{h}\|$ for the multiscale method (blue, dashed line) and $\|\tilde{\sigma}_{h,n}-\hat{\sigma}_{h}\|/\|\hat{\sigma}_{h}\|$ for the single-step regularization method (red, solid line), with respect to the iteration number. Left: Nebula NGC 1952. Right: Nebula NGC 7027.}
\end{figure}
\label{fig:4}
\end{center}

\section{Counterexamples}\label{examplesec}

In this section we construct examples showing that the tighter multiscale version can be needed even in the linear case. Overall we present three counterexamples, the first two in a discrete settings, namely for sequences spaces, the third in a continuous setting for the reconstruction of a deformed signal.

\subsection{First counterexample for sequences spaces}

Let $X=E=l_1$ and $H=l_2$. We consider
$$F=\tilde{E}=\left\{\gamma\in l_1:\ \|\gamma\|_F=\sum_{n=1}^{+\infty}n|\gamma_n|<+\infty\right\}.$$
We define $R$ as in \eqref{banachregular}. We note that $R$ is convex on $X$, it satisfies Assumption~\ref{strongerassum}, and we also have $F\subset X$.
Note that
$$G=\left\{\kappa:\ \|\kappa\|_{\ast}=\sup_{n\in\mathbb{N}}\frac{|\kappa_n|}{n}<+\infty\right\}$$
where the duality is given by
$$\langle \kappa,\gamma\rangle=\sum_{n=1}^{+\infty}\kappa_n\gamma_n\quad\text{for any }\kappa\in G\text{ and }\gamma\in F.$$
We fix $\alpha=2$ and $\beta=1$.

We fix constants $M\geq 2$, $\alpha_0>0$, $k\in \mathbb{N}$, $k\geq 3$, $c_0>0$, $b>0$, $0<\delta<1/k$. We call $A_1=1-k\delta$, $A_i=\delta$, $i=2,\ldots,k-1$, and $A_k=2\delta$, so that $\displaystyle \sum_{i=1}^kA_i=1$. For any $m=1,\ldots,k$, we call $\displaystyle C_m=\sum_{i=1}^mA_i$. Clearly $C_1=A_1$ and $C_k=1$.

For any $n\geq 0$ we let
$$\lambda_n=\alpha_0M^n.$$
For any $j\in \mathbb{N}$, let
$$\eta_j=\sqrt{\frac{c_0}{M^j}}\text{ and }\mu_j=-\delta\eta_j.$$
For any $j\geq 0$, let
\begin{equation}\label{ujdef}
u_j=\frac{b}{j+2}e_{j+2}
\end{equation}
and, for any $n\geq 0$, let
\begin{equation}\label{sigmandef}
\sigma_n=\sum_{j=0}^nu_j.
\end{equation}
We note that
$\displaystyle \lim_n\|\sigma_n\|_{X}=+\infty$.

We define the linear operator $\Lambda:X\to H$ as follows.
For any $j\geq 2$
$$\Lambda(e_j)=\frac{j}{b}\left[\left(\sum_{i=1}^kA_i\eta_{j+i-2}e_{j+i-2} \right)+\mu_{j+k-2}e_{j+k-2}  - \mu_{j+k-1}e_{j+k-1}\right]$$
whereas
$$\Lambda(e_1)=\left(\sum_{m=1}^kC_m\eta_me_m\right)+\left(\sum_{m=k+1}^{+\infty}\eta_ne_m\right)  +\mu_{k}e_k.$$
Clearly
\begin{equation}\label{Lambdadefinition}
\Lambda(\gamma)=\sum_{j=1}^{+\infty}\gamma_j\Lambda(e_j)\quad\text{for any }\gamma\in l_1.
\end{equation}

\begin{lem}\label{injlemma}
The operator $\Lambda$ defined in \eqref{Lambdadefinition} is a bounded linear operator from $X=l_1$ into $H=l_2$. Moreover, 
$\Lambda$ is injective on $l_1$, thus in particular on $F$.
\end{lem}

\begin{proof}
It is immediate to check that $\Lambda$ is a bounded linear operator from $l_1$ to $l_2$, actually it is a bounded linear operator from $l_p$ to $l_2$ for any $1\leq p\leq +\infty$. We now prove that $\Lambda$ is injective on $l_1$. Let us assume that
$\tilde{\gamma}=\Lambda(\gamma)=0$, that is, $\tilde{\gamma}_j=0$ for any $j\in\mathbb{N}$. We call $b_j=b/j$, $j\geq 2$, and we claim that $\Lambda(\gamma)=0$ if and only if
$$\frac{\gamma_j}{b_j}=-\gamma_1\quad\text{for any }j\geq 2.$$
It follows that $\gamma_1=0$, thus $\gamma=0$, otherwise $\gamma=\gamma_1(1,-b/2,\ldots,-b/j,\ldots)\not\in l^1$.
We prove the claim by induction. We have
$$0=\tilde{\gamma}_1=A_1\eta_1\gamma_1+A_1\eta_1\frac{\gamma_2}{b_2},$$
that is, $\gamma_2/b_2=-\gamma_1$. Let $j\geq 3$ and assume that $\gamma_l/b_l=-\gamma_1$ for any $2\leq l<j$. For $2\leq m=j-1<k$,
$$0=\tilde{\gamma}_m=C_{m}\eta_m\gamma_1+\sum_{i=1}^mA_i\eta_m\frac{\gamma_{m+2-i}}{b_{m+2-i}}=\eta_m\left(C_m\gamma_1-\sum_{i=2}^mA_i\gamma_1+ A_1\frac{\gamma_j}{b_j}\right)$$
and the claim is true for $2\leq j\leq k$. For $m=k$ we have
\begin{multline*}
0=\tilde{\gamma}_k=(\eta_k+\mu_k)\gamma_1+(A_k\eta_k+\mu_k)\frac{\gamma_2}{b_2}+\sum_{i=1}^{k-1}A_i\eta_k\frac{\gamma_{k+2-i}}{b_{k+2-i}}\\=
\eta_k\left(C_{k-1}\gamma_1-\sum_{i=2}^{k-1}A_i\gamma_1+A_1\frac{\gamma_{k+1}}{b_{k+1}}\right)
\end{multline*}
and the claim holds for $2\leq j\leq k+1$. Let $j-1=m\geq k+1$. Then
\begin{multline*}
0=\tilde{\gamma}_{m}=\eta_{m}\gamma_1+\sum_{i=1}^{k}A_i\eta_m\frac{\gamma_{m+2-i}}{b_{m+2-i}}+\mu_m\left[\frac{\gamma_{m+2-k}}{b_{m+2-k}}-\frac{\gamma_{m+1-k}}{b_{m+1-k}}\right]\\=\eta_m\left(C_k\gamma_1-\sum_{i=2}^{k}A_i\gamma_1+A_1\frac{\gamma_{m+1}}{b_{m+1}}\right)
\end{multline*}
and the proof is concluded.\end{proof}

\bigskip

We fix $\tilde{\Lambda}=\Lambda(e_1)$. Hence Assumption~\ref{stronger} is satisfied. In fact, $e_1\in F$ and, by the injectivity of $\Lambda$, $e_1$ is the only solution to
$$0=\|\tilde{\Lambda}-\Lambda(e_1)\|_H=\min\{\||\tilde{\Lambda}-\Lambda(\gamma)\|_H:\ \gamma\in l_1\}.$$
Nevertheless, we have the following result.
\begin{teo}\label{examplethm}
For any $M\geq 2$ and $\alpha_0>0$, if $k=5$,
there exist suitable constants $c_0>0$, $b>0$, $0<\delta<1/k$ such that the sequence $\{\sigma_n\}_{n\geq 0}$ as in \eqref{ujdef}, \eqref{sigmandef} coincides with the
multiscale sequence $\{\sigma_n\}_{n\geq 0}$ given by \eqref{regularizedpbm}, \eqref{inductiveconstr} and \eqref{partialsum1}.
\end{teo}

\begin{oss}
A simpler version of this example, which illustrates the underlying idea without all technicalities, is presented in \cite{R-R-IPMS}. However, this simpler version does not work for any $M\geq 2$ but only for $M$ big enough, roughly for $M\geq 6$.
\end{oss}

\begin{oss}
We note that, by injectivity of $\Lambda$ and Remark~\ref{uniquenessremark}, the multiscale sequence is uniquely defined. Since $\|\sigma_n\|_X\to +\infty$ as $n\to +\infty$, we have
an example where no subsequence of $\sigma_n$ converges, despite we are in the linear case and 
both Assumptions~\ref{strongerassum} and \ref{stronger} hold.
Let us note, however, that we can also consider $\Lambda$ as a bounded linear operator from $l_2$ to $l_2$. In this case, $\Lambda$ is not injective anymore and actually
$\displaystyle \sigma_n\to \sigma_{\infty}=\sum_{j=2}\frac{b}{j}e_j$, where the series is in the $l_2$ sense, and, by the arguments developed in the proof of Lemma~\ref{injlemma}, we have $\Lambda(\sigma_{\infty})=\Lambda(e_1)$.

Another interesting fact is that, on the contrary, in this case the single-step regularization actually works. Namely, for $\lambda>0$ we call $\tilde{\sigma}_{\lambda}$ the solution to
$$\min
\left\{\lambda\|\Lambda(e_1)-\Lambda(\sigma)\|_H^2+\|\sigma\|_{F}:\ \sigma\in X\right\},$$
and we have that $\tilde{\sigma}_{\lambda}$ is uniquely defined and it is not difficult to show that indeed
$$\lim_{\lambda\to +\infty}\|\Lambda(e_1)-\Lambda(\tilde{\sigma}_{\lambda})\|_{H}=0\quad\text{and}\quad\lim_{\lambda\to +\infty}\|e_1-\tilde{\sigma}_{\lambda}\|_{X}=0.
$$

\end{oss}

\begin{proof} By the characterization of minimizers developed in Subsection~\ref{mincarsec}, in particular by Proposition~\ref{charprop}, it is enough to show that
for any $n\geq 0$ we have
\begin{multline}\label{tobeproved}
\|(\Lambda(e_1)-\Lambda(\sigma_n))\circ\Lambda\|_{\ast}\leq \dfrac1{2\lambda_n}\quad\text{and}\\\langle \Lambda(e_1)-\Lambda(\sigma_n),\Lambda(u_n)\rangle_{H}=\dfrac1{2\lambda_n}\|u_n\|_F,
\end{multline}
that is, that \eqref{char} holds for any $n\geq 0$.

For any $n\geq 0$, we call $n_1=n+2\geq 2$ and compute $\displaystyle \Lambda(\sigma_n)=\sum_{j=2}^{n_1}\frac{b}{j}\Lambda(e_j)$.
We have, for $2\leq n_1\leq k$,
\begin{multline*}
\Lambda(\sigma_n)=\left(
\sum_{m=1}^{n_1-1}C_m\eta_me_m\right) + \left(\sum_{m=n_1}^k\left(\sum_{i=1}^{n_1-1}A_{m+1-i}\right)\eta_me_m\right)\\+
\left(\sum_{m=k+1}^{n_1+k-2}\left(1-C_{m+1-n_1}\right)\eta_me_m\right)
+\mu_ke_k-\mu_{n_1+k-1}e_{n_1+k-1}
\end{multline*}
and for $n_1\geq k+1$
\begin{multline*}
\Lambda(\sigma_n)=\left(
\sum_{m=1}^{k-1}C_m\eta_me_m\right) + \left(\sum_{m=k}^{n_1-1}\eta_me_m\right)\\+
\left(\sum_{m=n_1}^{n_1+k-2}\left(1-C_{m+1-n_1}\right)\eta_me_m\right)
+\mu_ke_k-\mu_{n_1+k-1}e_{n_1+k-1}.
\end{multline*}
We conclude that, in both cases, that is, for any $n_1\geq 2$,
\begin{multline*}
\Lambda(e_1)-\Lambda(\sigma_n)   =\left(\sum_{m=n_1}^{n_1+k-2}C_{m+1-n_1}\eta_me_m\right)+
(\eta_{n_1+k-1}+\mu_{n_1+k-1})e_{n_1+k-1}
  \\+\left(\sum_{m=n_1+k}^{+\infty}\eta_me_m\right).
\end{multline*}
For any $j\in \mathbb{N}$ and $n_1=n+2\geq 2$, we call
$$A_{j,n_1}=\frac1j\langle \Lambda(e_1)-\Lambda(\sigma_n),\Lambda(e_j)\rangle_{H}$$
so that
$$\|(\Lambda(e_1)-\Lambda(\sigma_n))\circ\Lambda\|_{\ast}=\sup_{j\in\mathbb{N}}|A_{j,n_1}|.$$
Our aim is to show that the following claim is true.

\begin{claim}\label{crucialclaim} 
For any $M\geq 2$ and $\alpha_0>0$, if $k=5$,
there exist suitable constants $c_0>0$, $b>0$, $0<\delta<1/k$ such that
the following holds.
 For any $n\geq 0$, calling $n_1=n+2$, we have
\begin{equation}\label{crucialclaimprop}
|A_{j,n_1}|\leq A_{n_1,n_1}=\frac{1}{2\lambda_n}=\frac1{2\alpha_0M^{n}}=\frac{M^2}{2\alpha_0M^{n_1}}\quad\text{for any }j\in\mathbb{N}.
\end{equation}
More precisely,
let $j(n_1)=\min\{j:\ 2\leq j\leq n_1: A_{j,n_1}\neq 0\}$. We claim that 
\begin{multline}\label{A_jincrea}
0<A_{1,n_1}<A_{j(n_1),n_1},\quad A_{j,n_1}<A_{j+1,n_1}\text{ for any }j(n_1)\leq j<n_1\\\text{and}\quad
0<A_{1,n_1}+A_{j(n_1),n_1}< A_{j(n_1)+1,n_1}\text{ if }j(n_1)<n_1
\end{multline}
and
\begin{equation}\label{A_jdecrea}
0<A_{j+1,n_1}<A_{j,n_1}\quad\text{for any }j\geq n_1.
\end{equation}
\end{claim}

Indeed,
\eqref{tobeproved} immediately follows from Claim~\ref{crucialclaim}, hence the proof of Theorem~\ref{examplethm} would be concluded.

It remains to prove Claim~\ref{crucialclaim}. We begin by computing $A_{j,n_1}$. First,
\begin{multline*}
A_{1,n_1}=\left(\sum_{m=n_1}^{k-1}C_{m+1-n_1}C_m\eta^2_m\right)+
C_{k+1-n_1}\eta_{k}(\eta_{k}+\mu_{k})\\+\left(\sum_{m=k+1}^{n_1+k-2}C_{m+1-n_1}\eta^2_m\right)+
(\eta_{n_1+k-1}+\mu_{n_1+k-1})\eta_{n_1+k-1}
 \\ +\left(\sum_{m=n_1+k}^{+\infty}\eta^2_m\right)\quad\text{for }2\leq n_1<k,
\end{multline*}
and
\begin{multline*}
A_{1,n_1}=C_1\eta_{n_1}(\eta_{n_1}+\mu_{n_1})+\left(\sum_{m=n_1+1}^{n_1+k-2}C_{m+1-n_1}\eta^2_m\right)\\+
(\eta_{n_1+k-1}+\mu_{n_1+k-1})\eta_{n_1+k-1}
  +\left(\sum_{m=n_1+k}^{+\infty}\eta^2_m\right)\quad\text{for }n_1=k,
\end{multline*}
and, finally,
\begin{multline*}
A_{1,n_1}=\left(\sum_{m=n_1}^{n_1+k-2}C_{m+1-n_1}\eta^2_m\right)+
(\eta_{n_1+k-1}+\mu_{n_1+k-1})\eta_{n_1+k-1}
  \\+\left(\sum_{m=n_1+k}^{+\infty}\eta^2_m\right)\quad\text{for }n_1>k.
\end{multline*}
We  note that in all cases, that is, for any $n_1\geq 2$
$$0< A_{1,n_1}\leq \sum_{m=n_1}^{+\infty}\eta_m^2=\frac{c_0}{M^{n_1}}\left(\frac{M}{M-1}\right).$$

We have
$$A_{j,n_1}=0\quad\text{for any }2\leq j\leq n_1-k$$
from which $j(n_1)$ as in the statement of the claim is equal to $\max\{2,n_1-k+1\}$.

For $j=n_1+s-k$, with $1\leq s<k$ and $j\geq 2$, we have
\begin{multline*}
bA_{n_1+s-k,n_1}=\left(\sum_{m=n_1}^{n_1+s-3}C_{m+1-n_1}A_{2+k-s+m-n_1}\eta_m^2\right)\\+C_{s-1}\eta_{n_1+s-2}(A_k\eta_{n_1+s-2}+\mu_{n_1+s-2})
-C_s\mu_{n_1+s-1}\eta_{n_1+s-1}.
\end{multline*}
Hence
$$
 A_{n_1+s-k,n_1}=\frac{c_0\delta}{bM^{n_1}}\left(\sum_{m=1}^{s}\frac{C_{m}}{M^{m-1}}\right).
$$
The most important term is for $j=n_1$
\begin{multline*}
bA_{n_1.n_1}=\left(\sum_{m=n_1}^{n_1+k-3}C_{m+1-n_1}A_{2+m-n_1}\eta_m^2\right)\\+C_{k-1}\eta_{n_1+k-2}(A_k\eta_{n_1+k-2}+\mu_{n_1+k-2})
-\mu_{n_1+k-1}(\eta_{n_1+k-1}+\mu_{n_1+k-1}).
\end{multline*}
We have that
\begin{multline*}
A_{n_1.n_1}=\frac{c_0}{bM^{n_1}}\left[\left(\sum_{m=0}^{k-3}\frac{C_{m+1}A_{2+m}}{M^{m}}\right)+C_{k-1}\frac{A_k-\delta}{M^{k-2}}+\frac{\delta(1-\delta)}{M^{k-1}}\right]
\\=\frac{c_0\delta}{bM^{n_1}}\left(\sum_{m=1}^{k-1}\frac{C_m}{M^{m-1}}+\frac{1-\delta}{M^{k-1}}\right)
=\frac{c_0\delta}{bM^{n_1}}K
=\frac{M^2}{2\alpha_0M^{n_1}}
\end{multline*}
hence $c_0$ has to be defined as
\begin{equation}
c_0=\frac{bM^2}{2\alpha_0\delta K}.
\end{equation}
Note that $K>0$ depends on $M$, $k$ and $\delta$ only.
It is immediate to check that $0\leq A_{j,n_1}\leq A_{n_1,n_1}$ for any $2\leq j<n_1$ and that the second condition in \eqref{A_jincrea} is satisfied.
Moreover, the first and third conditions in \eqref{A_jincrea}, which immediately imply $0\leq A_{1,n_1}\leq A_{n_1,n_1}$, are satisfied
provided we fix $b$ such that
$$\frac{M}{M-1}<\frac{\delta(1-k\delta)}{bM^{k-1}},\text{ that is, }b=\frac{(M-1)\delta(1-k\delta)}{2M^{k}}.$$
%$$\frac{M}{M-1}\leq \frac{\delta}{b}K,\text{ that is, }0<b\leq \frac{\delta K(M-1)}{M}.$$
Hence, once $k$ and $\delta$ are fixed, first $b$ and then $c_0$ can be chosen as required. The choice of $k$ and $\delta$ is crucial to compare
$A_{n_1,n_1}$ with $A_{j,n_1}$ for $j>n_1$, which we compute next.
For $j=n_1+1$ we have
\begin{multline*}
bA_{n_1+1,n_1}=\left(\sum_{m=n_1}^{n_1+k-2}A_{m+1-n_1}C_{m+1-n_1}   \eta_m^2\right) \\+
 (A_k-\delta)(1-\delta)\eta^2_{n_1+k-1}+\delta\eta^2_{n_1+k}
\end{multline*}
therefore
$$
A_{n_1+1,n_1}=\frac{c_0}{bM^{n_1}}\left(A_1^2+\delta\sum_{m=1}^{k-2}\frac{C_{m+1}}{M^{m}}+
 \frac{\delta(1-\delta)}{M^{k-1}}+\frac{\delta}{M^k}\right).
$$
We recall that 
$$A_{n_1,n_1}=\frac{c_0}{bM^{n_1}}\left(A_1\delta+\delta\sum_{m=1}^{k-2}\frac{ C_{m+1}}{M^{m}}+\frac{\delta(1-\delta)}{M^{k-1}}\right),$$
hence we need to show that, for some $0<A_1<1$ and $0<\delta=(1-A_1)/k<1/k$, we have
$$A_1\delta= \frac{A_1(1-A_1)}{k}> A_1^2+ \frac{\delta}{M^k}=A_1^2+ \frac{1-A_1}{kM^k},$$
that is, multiplying by $k$,
$$ (k+1)A_1^2+\left(-1-\frac{1}{M^k}\right)A_1+\frac{1}{M^k}< 0.$$
For $0<A_1=\frac{1+\frac{1}{M^k}}{2(k+1)}<1$, the previous polynomial is negative if and only if
$$1+\frac{1}{M^k}> 2\sqrt{\frac{k+1}{M^k}}\quad\text{hence when }M^{k/2}> \sqrt{k}+\sqrt{k+1}.$$
But for $k=5$ we have
$$M^{5/2}\geq 2^{5/2}>\sqrt{5}+\sqrt{6}.$$
We conclude that, for $k=5$ and 
$$A_1=\frac{1+\frac{1}{M^5}}{12}\text{ or, equivalently, }
\delta=\frac{11-\frac{1}{M^5}}{60}$$
we have $A_{n_1+1,n_1}<A_{n_1,n_1}$. We also note that $A_1<\delta$.

For $j=n_1+s$, $2\leq s\leq k$, we have
\begin{multline*}
bA_{j,n_1}=\left(\sum_{m=n_1+s-1}^{n_1+k-2}A_{m+2-n_1-s}C_{m+1-n_1}   \eta_m^2\right) +  A_{k+1-s}(1-\delta)\eta^2_{n_1+k-1}    +\\
+\left(\sum_{m=n_1+k}^{n_1+s+k-3}A_{m+2-n_1-s}\eta^2_m\right)
+
(A_k\eta_{j+k-2}+\mu_{j+k-2})\eta_{j+k-2}-\mu_{j+k-1}\eta_{j+k-1},
\end{multline*}
that is,  for $j=n_1+s$, $2\leq s\leq k-2$,
$$
A_{j,n_1}\\=\frac{c_0}{bM^{n_1}}\left(\frac{A_1C_s}{M^{s-1}}+\sum_{m=s}^{k-2}\frac{\delta C_{m+1}}{M^m}+\frac{\delta(1-\delta)}{M^{k-1}}+\sum_{m=k}^{k+s-1}\frac{\delta}{M^m}\right).
$$
and
$$
A_{n_1+k-1,n_1}\\=\frac{c_0}{bM^{n_1}}\left(\frac{A_1C_{k-1}}{M^{k-2}}+\frac{\delta(1-\delta)}{M^{k-1}}+\sum_{m=k}^{2k-2}\frac{\delta}{M^m}\right).
$$
and
$$A_{n_1+k,n_1}=\frac{c_0}{bM^{n_1}}\left(\frac{A_1(1-\delta)}{M^{k-1}}+\sum_{m=k}^{2k-1}\frac{\delta}{M^m}\right).
$$
Let us compare $A_{n_1+s,n_1}$ with $A_{n_1+s+1,n_1}$ for $1\leq s\leq k-2$. We have
\begin{multline*}
\frac{bM^{n_1+s-1}}{c_0}\left(A_{n_1+s,n_1}-A_{n_1+s+1,n_1}\right)\\=A_1C_s+\frac{(\delta-A_1)C_{s+1}}{M}-\frac{\delta}{M^{k+1}}
> A_1^2-\frac{\delta}{M^{k+1}}
\end{multline*}
since $C_s\geq A_1$ for any $1\leq s \leq k-2$ and $A_1<\delta$. Since
$$A_1^2>\frac{1}{144}>\frac{11}{60}\frac{1}{2^6}> \frac{\delta}{M^6}$$
we conclude that for any $1\leq s\leq k-2$
\begin{equation}\label{almostlast}
A_{n_1,n_1}>A_{n_1+s,n_1}>A_{n_1+s+1,n_1}.
\end{equation}
For $s=k-1$,
\begin{multline*}
\frac{bM^{n_1+k-2}}{c_0}\left(A_{n_1+k-1,n_1}-A_{n_1+k,n_1}\right)\\=A_1C_{k-1}+\frac{(\delta-A_1)(1-\delta)}{M}-\frac{\delta}{M^{k+1}}
> A_1^2-\frac{\delta}{M^{k+1}}
\end{multline*}
and we infer that \eqref{almostlast} holds for any $1\leq s\leq k-1$.

For $j\geq n_1+k+1$
$$bA_{j,n_1}=\left(\sum_{m=j-1}^{j+k-3}A_{m+2-j}   \eta^2_m\right) +  (A_k\eta_{j+k-2}+\mu_{j+k-2})\eta_{j+k-2}-\mu_{j+k-1}\eta_{j+k-1}$$
hence
$$A_{j,n_1}=\frac{c_0}{bM^{n_1}}\left(\frac{A_1}{M^{j-n_1-1}}+ \sum_{m=j-n_1}^{j-n_1+k-1}\frac{\delta}{M^{m}} \right)\quad\text{for }j\geq n_1+k+1.
$$
It is immediate to see that for any $j\geq n_1+k+1$ we have
$$A_{j,n_1}= M A_{j+1,n_1}> A_{j+1,n_1}.$$
 Hence to prove \eqref{A_jdecrea} it is enough to show that
$A_{n_1+k,n_1}>A_{n_1+k+1,n_1}$.
We have
\begin{multline*}
\frac{bM^{n_1+k-1}}{c_0}\left(A_{n_1+k,n_1}-A_{n_1+k+1,n_1}\right)\\=A_1(1-\delta)+\frac{(\delta-A_1)}{M}-\frac{\delta}{M^{k+1}}
> A_1^2-\frac{\delta}{M^{k+1}}
\end{multline*}
since $A_1+\delta<1$. We have shown \eqref{A_jdecrea}, thus the proof is concluded.\end{proof}

\subsection{Second counterexample for sequences spaces}
A simple rescaling argument leads to the following case.
Let $X=E=H=l_2$. We consider
\begin{equation}\label{F2defin}
F=\tilde{E}=F_2=\left\{\gamma\in l_2:\ \|\gamma\|_{F_2}=\sum_{n=1}^{+\infty}\sqrt{n}|\gamma_n|<+\infty\right\}.
\end{equation}

For any $j\geq 0$, let
\begin{equation}\label{ujdefsecond}
u_j=\frac{b}{\sqrt{j+2}}e_{j+2}
\end{equation}
and, for any $n\geq 0$, let
\begin{equation}\label{sigmandefsecond}
\sigma_n=\sum_{j=0}^nu_j.
\end{equation}
We note that
$\displaystyle \lim_n\|\sigma_n\|_{X}=+\infty$.

Using the same notation as before,
we define the linear operator $\Lambda:X\to H$ as follows.
For any $j\geq 2$
\begin{multline}\label{2operatorj>1}
\Lambda(e_j)\\=\frac{\sqrt{j}}{b}\left[\left(\sum_{i=1}^kA_i\eta_{j+i-2}e_{j+i-2} \right)+\mu_{j+k-2}e_{j+k-2}  - \mu_{j+k-1}e_{j+k-1}\right]
\end{multline}
whereas
\begin{equation}\label{2operatorj=1}
\Lambda(e_1)=\left(\sum_{m=1}^kC_m\eta_me_m\right)+\left(\sum_{m=k+1}^{+\infty}\eta_me_m\right)  +\mu_{k}e_k.
\end{equation}
Clearly
\begin{equation}\label{Lambdadefinitionsecond}
\Lambda(\gamma)=\sum_{j=1}^{+\infty}\gamma_j\Lambda(e_j)\quad\text{for any }\gamma\in l_2.
\end{equation}

As in Lemma~\ref{injlemma}, we have that
the operator $\Lambda$ defined in \eqref{Lambdadefinitionsecond} is a bounded linear operator from $X=l_2$ into $H=l_2$. Moreover, 
$\Lambda$ is injective on $l_2$, thus in particular on $F=F_2$. Then, for $\tilde{\Lambda}=\Lambda(e_1)$, in a completely analogous way one can prove the corresponding version of Theorem~\ref{examplethm}.
\begin{teo}\label{examplethmsecond}
For any $M\geq 2$ and $\alpha_0>0$, if $k=5$,
there exist suitable constants $c_0>0$, $b>0$, $0<\delta<1/k$ such that the sequence $\{\sigma_n\}_{n\geq 0}$ as in \eqref{ujdefsecond}, \eqref{sigmandefsecond} coincides with the
multiscale sequence $\{\sigma_n\}_{n\geq 0}$ given by \eqref{regularizedpbm}, \eqref{inductiveconstr} and \eqref{partialsum1}.
\end{teo}

We just make a few remarks. If, for any $j\in \mathbb{N}$ and $n_1=n+2\geq 2$, we call
$$A_{j,n_1}=\frac1{\sqrt{j}}\langle \Lambda(e_1)-\Lambda(\sigma_n),\Lambda(e_j)\rangle_{H},$$
then
$$\|(\Lambda(e_1)-\Lambda(\sigma_n))\circ\Lambda\|_{\ast}=\sup_{j\in\mathbb{N}}|A_{j,n_1}|.$$
Moreover, the numbers $A_{j,n_1}$ coincide with those of the first example, hence they satisfy Claim~\ref{crucialclaim} from which the proof of Theorem~\ref{examplethmsecond}
immediately follows.

\subsection{A not converging multiscale decomposition for a bounded linear blurring operator on the line}
We call
$$e_j=(-1)^j\frac{\sqrt{j}}{2}\chi_{I_j}\quad \text{for any }j\in\mathbb{N}.$$
where $\{I_j=[a_j,b_j]\}_{j\in\mathbb{N}}$ is a sequence of closed pairwise disjoint intervals such that $a_1=0$ and
$$0\leq a_j<b_j=a_j+\frac{4}{j}<a_{j+1}=b_j+j^2\quad\text{for any }j\geq 1.$$
We call $L_j=(b_j,a_{j+1})$ and we have $|I_j|=4/j$ and $|L_j|=j^2$ for any $j\in\mathbb{N}$.

We have that $\{e_j\}_{j\in\mathbb{N}}$ is an orthonormal system in $L^2(\mathbb{R})$
and for any $\gamma\in l_2$ we have
$$\left\|\sum_{j=1}^{+\infty}\gamma_je_j\right\|_{L^2(\mathbb{R})}=\|\gamma\|_{l_2}\quad\text{and}\quad \left\|\sum_{j=1}^{+\infty}\gamma_je_j\right\|_{BV(\mathbb{R})}=
\|\gamma\|_{F_2}$$
where $F_2$ is as in \eqref{F2defin} and $\|\cdot\|_{BV(\R)}$ is the norm of the homogeneous $BV(\R)$ space, that is, it coincides with the total variation.
We call $\tilde{H}$ the closure of the span of $\{e_j\}_{j\in\mathbb{N}}$, that is,
$$u\in \tilde{H}\quad\text{if and only if}\quad u=\sum_{j=1}^{+\infty}\gamma_je_j\text{ for some }\gamma\in l_2.$$
We call
$$K=\Big\{v\in L^2(\mathbb{R}):\ v=\sum_{j=1}^{+\infty}\kappa_j\chi_{L_j}\text{ for some }\kappa\in l_{\infty}\Big\}$$
and note that $K$ is orthogonal to $\tilde{H}$. We call
$$P=(\tilde{H}+K)^{\perp}$$
hence any $w\in L^2(\mathbb{R})$ can be written in a unique way as
$$w=u+v+z\quad\text{with }u\in \tilde{H},\ v\in K\ \text{and}\ z\in P.$$
We note that if $z\in P$, then
$$\int_{I_j}z=\int_{L_j}z=0\quad\text{for any }j\in\mathbb{N}.$$

Let us investigate the behaviour of the total variation for such a decomposition. We use the following notation.
For any $a \in\R$ and for a (good representative of a) $BV$ function $w$, we denote $\displaystyle a^{\pm}=\lim_{x\to a^{\pm}}w(x)$.
We have the following lemma whose proof is immediate.
\begin{lem}\label{BVdecomposition}
Let $z\in L^2((-\infty,0))$. Then
$$|Dz|((-\infty,0))\geq |z(0^-)|.$$
Let $I$ an open interval and let $z\in L^2(I)$ be such that 
$\int_Iz=0$. Then
$$|Dz|(I)\geq |z^+(a)|+|z^-(b)|.$$
Finally, we have, for $w=u+v+z\in L^2(\R)$,
\begin{multline*}
\|u+v+z\|_{BV(\mathbb{R})}\geq
\bigg|-\gamma_1\frac{\sqrt{1}}{2}+z(a_1^+)-z(a_1^-)\bigg|\\
+\sum_{j=1}^{+\infty}\left[|z(a_j^-)|+|z(a_j^+)|+|z(b_j^-)|+|z(b_j^+)|\right]\\+
\sum_{j=1}^{+\infty}\bigg|(-1)^j\gamma_j\frac{\sqrt{j}}{2}+z(b_j^-)-\kappa_j-z(b_j^+)\bigg|\\+
\sum_{j=1}^{+\infty}\bigg|\kappa_j+z(a_{j+1}^-)-(-1)^{j+1}\gamma_{j+1}\frac{\sqrt{j+1}}{2}-z(a_{j+1}^+)\bigg|\\
\geq \|u+v\|_{BV(\mathbb{R})}
\geq
\sum_{j=1}^{+\infty}|\gamma_j|\sqrt{j}-
2\sum_{j=1}^{+\infty}|k_j|.
\end{multline*}
\end{lem}

We define $\Lambda:L^2(\mathbb{R})\to L^2(\mathbb{R})$ as follows. We define, for any $j\in \mathbb{N}$,
$\Lambda(e_j)$
as in \eqref{2operatorj=1} and \eqref{2operatorj>1}, hence
\begin{equation}\label{Lambdadefinitionthird}
\Lambda(u)=\sum_{j=1}^{+\infty}\gamma_j\Lambda(e_j)\quad\text{for any }u=\sum_{j=1}^{+\infty}\gamma_je_j\text{ with }\gamma\in l_2.
\end{equation}
Instead, we define
$$\Lambda(v+z)=v+z\quad\text{for any }v\in K\text{ and }z\in P,$$
that is,
$\Lambda|_{\tilde{H}^{\perp}}$ is the identity.

\begin{lem}\label{solvability}
We have that $\Lambda:L^2(\mathbb{R})\to L^2(\mathbb{R})$ is a bounded linear injective operator. Moreover, the problem
\begin{equation}\label{ROFmodified}
\min\{\lambda\|\tilde{\Lambda}-\Lambda(w)\|^2_{L^2(\mathbb{R})}+\|w\|_{BV(\mathbb{R})}:\ w
\in L^2(\mathbb{R})\}
\end{equation}
admits a unique solution for any $\lambda>0$ and any $\tilde{\Lambda}\in L^2(\mathbb{R})$.
\end{lem}

\begin{proof} First of all, we note that $\Lambda(u)\in \tilde{H}$ for any $u\in\tilde{H}$ and it coincides with the operator $\Lambda$ defined in our second variant, therefore
$\Lambda|_{\tilde{H}}:\tilde{H}\to \tilde{H}\subset L^2(\mathbb{R})$ is a bounded linear injective operator.
Since $\Lambda(\tilde{v})=\tilde{v}$
for any $\tilde{v}\in \tilde{H}^{\perp}$, we immediately conclude that
$\Lambda$ is a bounded linear and injective operator.

Hence, only the solvability of \eqref{ROFmodified} needs a proof. This is not completely trivial. Let $\tilde{\Lambda}=\tilde{u}+\tilde{v}+\tilde{z}$ and
let $w^n=u^n+v^n+z^n$, $n\in\mathbb{N}$, be a minimizing sequence. We note that the minimum problem can be written as
\begin{multline*}
\min\Big\{\lambda\|\tilde{u}-\Lambda(u)\|^2_{L^2(\mathbb{R})}+\lambda\|\tilde{v}-v\|^2_{L^2(\mathbb{R})}+\lambda\|\tilde{z}-z\|^2_{L^2(\mathbb{R})}
\\+\|u+v+z\|_{BV(\mathbb{R})}:\ w=u+v+z\in L^2(\mathbb{R})\Big\}
\end{multline*}
hence we can assume that for some constant $C$
$$\|u^n+v^n+z^n\|_{BV(\mathbb{R})},\,
\|v^n\|_{L^2(\mathbb{R})}+\|z^n\|_{L^2(\mathbb{R})}\leq C\quad\text{for any }n\in\mathbb{N}.$$

Hence, by Lemma~\ref{BVdecomposition},
$$C\geq \|u^n+v^n+z^n\|_{BV(\mathbb{R})}\geq \|u^n+v^n\|_{BV(\mathbb{R})}\geq \sum_{j=1}^{+\infty}|\gamma^n_j|\sqrt{j}-
2\sum_{j=1}^{+\infty}|k^n_j|.
$$
We have
\begin{multline*}
\sum_{j=1}^{+\infty}|k^n_j|= \sum_{j=1}^{+\infty}|k^n_j|\frac{j}{j}\\\leq \left(\sum_{j=1}^{+\infty}|k^n_j|^2j^{2}\right)^{1/2}\left(\sum_{j=1}^{+\infty}\frac{1}{j^{2}}\right)^{1/2}=
C_1\|v_n\|_{L^2(\mathbb{R})}\leq C_1C.
\end{multline*}
We conclude that
$$\|u_n\|_{L^2(\mathbb{R})}=\|\gamma\|_{l_2}\leq \|\gamma\|_{l_1}\leq \sum_{j=1}^{+\infty}|\gamma^n_j|\sqrt{j}
\leq C+2C_1C.$$
Therefore, for some constant $\tilde{C}$ we have for any $n\in\mathbb{N}$
$$\|w^n\|_{L^2(\mathbb{R})},\,\|w^n\|_{BV(\mathbb{R})} \leq \tilde{C}.$$
Then we argue like for the classical Rudin-Osher-Fatemi model. Namely, we can assume passing to subsequences that
$u^n$, $v^n$ and $z^n$ weakly converge in $L^2(\mathbb{R})$ to $u$, $v$ and $z$ respectively. By semicontinuity
$$\|\tilde{v}-v\|_{L^2(\mathbb{R})}\leq \liminf_n\|\tilde{v}-v^n\|_{L^2(\mathbb{R})}\quad\text{and}\quad
\|\tilde{z}-z\|_{L^2(\mathbb{R})}\leq \liminf_n\|\tilde{z}-z^n\|_{L^2(\mathbb{R})}$$
and
$$\|u+v+w\|_{BV(\mathbb{R})}\leq
\liminf_n \|u^n+v^n+w^n\|_{BV(\mathbb{R})}.$$
Since $\Lambda$ is linear and continuous, $u^n$ weakly converging to $u$ implies that $\Lambda(u^n)$ weakly converges to $\Lambda(u)$ and we immediately conclude that
$$\|\tilde{u}-\Lambda(u)\|^2_{L^2(\mathbb{R})}\leq\liminf_n\|\tilde{u}-\Lambda(u^n)\|^2_{L^2(\mathbb{R})}.$$
Hence $w=u+v+z$ is a minimizer. Uniqueness follows by Remark~\ref{uniquenessremark} and the proof is concluded.\end{proof}

\begin{oss}If $\tilde{z}=0$, that is, $\tilde{\Lambda}=\tilde{u}+\tilde{v}$, then, $w$, the solution to 
\eqref{ROFmodified}, is such that $z=0$, that is, $w=u+v$ where $u+v$ solve
\begin{multline}\label{ROFmodifiedbis}
\min\Big\{\lambda\|\tilde{u}-\Lambda(u)\|^2_{L^2(\mathbb{R})}+\lambda\|\tilde{v}-v\|^2_{L^2(\mathbb{R})}
+\|u+v\|_{BV(\mathbb{R})}:\\ w=u+v\in \tilde{H}+K\Big\}.
\end{multline}
\end{oss}

For any $j\geq 0$, let
\begin{equation}\label{ujdefthird}
u_j=\frac{b}{\sqrt{j+2}}e_{j+2}
\end{equation}
and, for any $n\geq 0$, let
\begin{equation}\label{sigmandefthird}
\sigma_n=\sum_{j=0}^nu_j.
\end{equation}

\begin{oss}
We obtain that $\displaystyle \|\sigma_n\|^2_{L^2(\mathbb{R})}=b^2\sum_{j=2}^{n+2}\frac1j$, hence $$\displaystyle \lim_n\|\sigma_n\|_{L^2(\mathbb{R})}=+\infty.$$
Moreover, $\displaystyle \|\sigma_n\|_{L^1(\mathbb{R})}=2b\sum_{j=2}^{n+2}\frac1{j}$ and
$\displaystyle \|\sigma_n\|_{BV(\mathbb{R})}=b(n+1)$ thus we also have
$$
\lim_n\|\sigma_n\|_{L^1(\mathbb{R})}=+\infty\quad\text{and}\quad \lim_n\|\sigma_n\|_{BV(\mathbb{R})}=+\infty.$$
\end{oss}

We fix $\tilde{\Lambda}=\Lambda(e_1)$. Our aim is to show that the following result holds. 
\begin{teo}\label{examplethmthird}
For any $M\geq 2$ and $\alpha_0>0$, if $k=5$,
there exist suitable constants $c_0>0$, $b>0$, $0<\delta<1/k$ such that the sequence $\{\sigma_n\}_{n\geq 0}$ as in \eqref{ujdefthird}, \eqref{sigmandefthird} coincides with the
multiscale sequence $\{\sigma_n\}_{n\geq 0}$ given by \eqref{regularizedpbm}, \eqref{inductiveconstr} and \eqref{partialsum1}.
\end{teo}

The remaining part of this section is devoted to the proof of Theorem~\ref{examplethmthird}.
We argue by induction. We call $\tilde{\Lambda}=\tilde{\Lambda}_{-1}$ and, for any $n\geq 0$,
$\tilde{\Lambda}_n=\tilde{\Lambda}-\Lambda(\sigma_n)$, where $\sigma_n$ is as in \eqref{sigmandefthird}. For any $n\geq 0$, we show that $u_n$ as in \eqref{ujdefthird}
is the unique minimizer for
$$\min\big\{F_n(w):\ w=u+v\in \tilde{H}+K\big\}$$
where $F_n(w)=\lambda_n\|\tilde{\Lambda}_{n-1}-\Lambda(w)\|^2_{L^2(\mathbb{R})}+\|w\|_{BV(\mathbb{R})}$.
By convexity, a necessary and sufficient condition is that the following kind of directional derivative is nonnegative for any direction, namely
$$\frac{\partial F_n}{\partial w}(u_n)=\limsup_{\varepsilon\to 0^+}\frac{F_n(u_n+\varepsilon w)-F_n(u_n)}{\varepsilon}\geq 0\quad \text{for any }w\in \tilde{H}+K.$$
Let $$w=u+v=\sum_{j=1}^{+\infty}\alpha_j\chi_{I_j}+\sum_{j=1}^{+\infty}\beta_j\chi_{L_j}.$$
Since $w\in L^2(\mathbb{R})$, we have that
\begin{equation}\label{liminf}
\sum_{j=1}^{+\infty}\frac{|\alpha_j|^2}{j}<+\infty,\text{ hence }\liminf_{j\to+\infty}|\alpha_j|=0.
\end{equation}
Then, with $n_1=n+2$, since $\|u_n\|_{BV(\R)}=b$ for any $n\geq 0$,
\begin{multline*}
F_n(u_n+\varepsilon w)-F_n(u_n)\\=-2\lambda_n\varepsilon\langle\tilde{\Lambda}_{n-1}-\Lambda(u_n),\Lambda(u)\rangle+\varepsilon^2\|\Lambda(w)\|_{L^2(\mathbb{R})}^2+\|u_n+\varepsilon w\|_{BV(\mathbb{R})}-b\\
=-2\lambda_n\varepsilon\langle\tilde{\Lambda}-\Lambda(\sigma_n),\Lambda(u)\rangle+\varepsilon^2\|\Lambda(w)\|_{L^2(\mathbb{R})}^2+\|u_n+\varepsilon w\|_{BV(\mathbb{R})}-b
\\=-4\lambda_n\varepsilon\sum_{j=1}^{+\infty}(-1)^{j}A_{j,n_1}\alpha_j+\varepsilon^2\left(\|\Lambda(u)\|_{L^2(\mathbb{R})}^2
+
\|v\|_{L^2(\mathbb{R})}^2\right)\\+\|u_n+\varepsilon w\|_{BV(\mathbb{R})}-b.
\end{multline*}
Hence, we can equivalently define
$$\frac{\partial F_n}{\partial w}(u_n)=\limsup_{\varepsilon\to 0^+}\frac{1}{\varepsilon}\bigg(-4\lambda_n\varepsilon\sum_{j=1}^{+\infty}(-1)^jA_{j,n_1}\alpha_j
+\|u_n+\varepsilon w\|_{BV(\mathbb{R})}-b\bigg)
$$
which is convenient since from now on we can drop the assumption that $v\in L^2(\mathbb{R})$ and keep only the one that $w\in BV(\mathbb{R})$.
This implies in particular that at least we have $\{\alpha_j\}_{j\in\N}\in l_{\infty}$ and $\{\beta_j\}_{j\in\N}\in l_{\infty}$.
We recall that $u_n=(-1)^{n_1}\dfrac{b}{2}\chi_{I_{n_1}}$ and we can assume that, for $0<\varepsilon\leq \varepsilon_0$,
$$\frac b2>2\varepsilon\max_{j\in\mathbb{N}}\{|\alpha_j|,|\beta_j|\}.$$
From now on, we always assume that $0<\varepsilon\leq \varepsilon_0$.
We have that, with the new definition,
$$\frac{\partial F_n}{\partial w}(u_n)\geq \frac{\partial F_n}{\partial \tilde{w}}(u_n)$$
for $\tilde{w}=\tilde{w}(u)=u+\tilde{v}(u)$ where
$\displaystyle \tilde{v}(u)=\sum_{j=1}^{+\infty}\tilde{\beta}_j\chi_{L_j}$ is chosen as follows.
First, since $$b/2-\varepsilon |\alpha_{n_1}|>\varepsilon\max_{j\in\mathbb{N}}\{|\alpha_j|,|\beta_j|\},$$
we set
$$\tilde{\beta}_{n_1-1}=\max\{\alpha_{n_1-1},0\}\quad\text{and}\quad\tilde{\beta}_{n_1}=\max\{\alpha_{n_1+1},0\},\quad\text{for }n_1\text{ even}$$
and
$$\tilde{\beta}_{n_1-1}=\min\{\alpha_{n_1-1},0\}\quad\text{and}\quad\tilde{\beta}_{n_1}=\min\{\alpha_{n_1+1},0\},\quad\text{for }n_1\text{ odd}.$$
For all remaining $j\in\mathbb{N}\backslash\{n_1-1,n_1\}$, the rule is the following. If
$\alpha_j\alpha_{j+1}\geq 0$, we pick $|\tilde{\beta}_j|=\min\{|\alpha_j|,|\alpha_{j+1}|\}$, with sign equal to the sign of the one, if any, which is not $0$. 
If $\alpha_j\alpha_{j+1}<0$, we pick $\tilde{\beta}_j=0$. In fact such a $\tilde{v}(u)$ has indeed the effect of minimizing $\|u_n+\varepsilon (u+v)\|_{BV(\mathbb{R})}$ for all
$\displaystyle v=\sum_{j=1}^{+\infty}\beta_j\chi_{L_j}$.
Hence, without loss of generality, in what follows we assume that $v=\tilde{v}(u)$ and note that we can drop completely the dependence on $v$ in the directional derivative, namely we call, with a slight abuse of notation
\begin{equation}\label{dirdernew}
\frac{\partial F_n}{\partial u}(u_n)=-4\lambda_n\sum_{j=1}^{+\infty}(-1)^jA_{j,n_1}\alpha_j
+\limsup_{\varepsilon\to 0^+}\frac{\|u_n+\varepsilon (u+\tilde{v}(u))\|_{BV(\mathbb{R})}-b}{\varepsilon}.
\end{equation}
Our aim is to show that for any $u\in L^2(\mathbb{R})$ with $u+\tilde{v}(u)\in BV(\mathbb{R})$, we have that
$\displaystyle \frac{\partial F_n}{\partial u}(u_n)$ defined as in \eqref{dirdernew} is nonnegative.

We begin by considering the parameter $\alpha_{n_1}$. If we consider $\hat{u}=u-\alpha_{n_1}\chi_{I_{n_1}}$
then
$\displaystyle \frac{\partial F_n}{\partial u}(u_n)= \frac{\partial F_n}{\partial \hat{u}}(u_n)$. In fact, first we note that
$\tilde{v}(u)=\tilde{v}(\hat{u})$. Then
$$-4\lambda_n\varepsilon(-1)^{n_1}A_{n_1,n_1}\alpha_{n_1}=(-1)^{n_1+1}2\varepsilon\alpha_{n_1}
$$
and $\|u_n+\varepsilon (u+\tilde{v}(u))\|_{BV(\mathbb{R})}=\|u_n+\varepsilon(\hat{u}+\tilde{v}(\hat{u}))\|_{BV(\mathbb{R})}+(-1)^{n_1}2\varepsilon\alpha_{n_1}$.
Therefore, in what follows we further assume without loss of generality that $u$ is such that $\alpha_{n_1}=0$.

We use the convention that $\tilde{\beta}_0=\tilde{\beta}_{\infty}=0$.
We consider four different cases.
\begin{enumerate}[i\textnormal{)}]
\item\label{case1}
 Let $n_1< j_0\leq j_1\leq +\infty $ be such that $\alpha_j$, for $j_0\leq j\leq  j_1$ or $j_0\leq j$ if $j_1=+\infty$, are either all positive or all negative and $\tilde{\beta}_{j_0-1}=\tilde{\beta}_{j_1}=0$.
\item\label{case2} Let $1\leq j_0\leq j_1<n_1$ be such that $\alpha_j$, for $j_0\leq j\leq  j_1$, are either all positive or all negative and $\tilde{\beta}_{j_0-1}=\tilde{\beta}_{j_1}=0$.
\item\label{case3} Let $j_0=n_1+1\leq j_1\leq+\infty$ and let us assume that for any $j_0\leq j\leq j_1$ or $j_0\leq j$ if $j_1=+\infty$, $(-1)^{n_1}\alpha_j>0$ and $\tilde{\beta}_{j_1}=0$.
\item\label{case4} Let $1\leq j_0\leq j_1=n_1-1$ and let us assume that for any $j_0\leq j\leq j_1$, $(-1)^{n_1}\alpha_j>0$ and $\tilde{\beta}_{j_0-1}=0$.
\end{enumerate}

We begin with the following lemma.
\begin{lem}\label{monotonelemma}
In all cases, we call $\hat{u}$ the function which is obtained from $u$ by setting $\alpha_j=0$ for any $j_0\leq j \leq j_1$
or $j_0\leq j$ if $j_1=+\infty$.

In cases \textnormal{\ref{case1})} and \textnormal{\ref{case3})}, assume $|\alpha_j|\geq |\alpha_{j+1}|$ for any $j_0\leq j<j_1$.

In cases \textnormal{\ref{case2})} and \textnormal{\ref{case4})}, assume $|\alpha_j|\leq |\alpha_{j+1}|$ for any $j_0\leq j<j_1$.

Then
$\displaystyle \frac{\partial F_n}{\partial u}(u_n)\geq \frac{\partial F_n}{\partial \hat{u}}(u_n)$,
directional derivatives defined as in \eqref{dirdernew}.
\end{lem}

\begin{proof} We call $w=u+\tilde{v}(u)$ and $\hat{w}=\hat{u}+\tilde{v}(\hat{u})$.

 Let us prove the first case.  We have that
$$\|u_n+\varepsilon w \|_{BV(\mathbb{R})}=\|u_n+\varepsilon \hat{w}\|_{BV(\mathbb{R})}+2\varepsilon |\alpha_{j_0}|,$$
using \eqref{liminf} when $j_1=+\infty$.
The contribution given by these $\alpha_j$, $j_0\leq j\leq j_1$ or $j_0\leq j$ if $j_1=+\infty$, is 
$$-4\lambda_n\varepsilon\sum_{j=j_0}^{j_1}(-1)^jA_{j,n_1}\alpha_j$$
and by \eqref{A_jdecrea}
$$\left|-4\lambda_n\varepsilon\sum_{j=j_0}^{j_1}(-1)^jA_{j,n_1}\alpha_j\right|\leq 4\lambda_n\varepsilon A_{j_0,n_1}  |\alpha_{j_0}| \leq 2\varepsilon |\alpha_{j_0}|$$
and the first case is proved.

About the second case, the argument is completely analogous. In fact,
we have that
$$\|u_n+\varepsilon w\|_{BV(\mathbb{R})}=\|u_n+\varepsilon \hat{w}\|_{BV(\mathbb{R})}+2\varepsilon |\alpha_{j_1}|.$$
The contribution given by these $\alpha_j$, $j_0\leq j\leq j_1$, is 
$$-4\lambda_n\varepsilon\sum_{j=j_0}^{j_1}(-1)^jA_{j,n_1}\alpha_j$$
and by \eqref{A_jincrea}
$$\left|-4\lambda_n\varepsilon\sum_{j=j_0}^{j_1}(-1)^jA_{j,n_1}\alpha_j\right|\leq 2\varepsilon |\alpha_{j_1}|.$$
Some extra care is required in this case. If $2\leq j_0$, then we easily obtain
$$\left|-4\lambda_n\varepsilon\sum_{j=j_0}^{j_1}(-1)^jA_{j,n_1}\alpha_j\right|\leq 4\lambda_n\varepsilon A_{j_1,n_1}  |\alpha_{j_1}|\leq 2\varepsilon |\alpha_{j_1}|.$$
If $j_0=1$ and $j_1> j(n_1)$, then
\begin{multline*}
\left|-4\lambda_n\varepsilon\sum_{j=j_0}^{j_1}(-1)^jA_{j,n_1}\alpha_j\right|=
\left|-4\lambda_n\varepsilon (-1) A_{1,n_1}\alpha_1    -4\lambda_n\varepsilon\sum_{j=j(n_1)}^{j_1}(-1)^jA_{j,n_1}\alpha_j\right|\\
\leq 4\lambda_n\varepsilon A_{j_1,n_1}  |\alpha_{j_1}|\leq 2\varepsilon |\alpha_{j_1}|
\end{multline*}
since we use $A_{1,n_1}< A_{j(n_1)}$ when $j(n_1)$ is even and $A_{1,n_1}+ A_{j(n_1)}<A_{j(n_1)+1}$ when $j(n_1)$ is odd.
The same argument works when $j_1= j(n_1)$ with $j(n_1)$ even.

If $j_0=1$ and $j_1= j(n_1)$ with $j(n_1)$ odd, then
$$
\left|-4\lambda_n\varepsilon\sum_{j=j_0}^{j_1}(-1)^jA_{j,n_1}\alpha_j\right|=4\lambda_n\varepsilon
\left| A_{1,n_1}\alpha_1 + A_{j(n_1),n_1}\alpha_{j(n_1)}\right|\\
\leq 2\varepsilon |\alpha_{j_1}|
$$
since we use $A_{1,n_1}+ A_{j(n_1)}<A_{j(n_1)+1}\leq 1/(2\lambda_n)$.

Finally, if $j_0=1$ and $j_1<j(n_1)$, then, since $A_{1,n_1}<1/(2\lambda_n)$,
$$
\left|-4\lambda_n\varepsilon\sum_{j=j_0}^{j_1}(-1)^jA_{j,n_1}\alpha_j\right|=4\lambda_n\varepsilon
\left| A_{1,n_1}\alpha_1 \right| \leq 2\varepsilon |\alpha_1|
\leq 2\varepsilon |\alpha_{j_1}|.
$$

Both in the third and fourth case, we have $\|u_n+\varepsilon w\|_{BV(\mathbb{R})}=\|u_n+\varepsilon \hat{w}\|_{BV(\mathbb{R})}$.
In the third case, $-4\lambda_n\varepsilon(-1)^{n_1+1}A_{n_1+1,n_1}\alpha_{n_1+1}>0$ and
the contribution given by
these $\alpha_j$, $j_0\leq j\leq j_1$ or $j_0\leq j$ if $j_1=+\infty$, is 
 $$-4\lambda_n\varepsilon\sum_{j=n_1+1}^{j_1}(-1)^jA_{j,n_1}\alpha_j$$
which is also positive by \eqref{A_jdecrea}.

In the fourth case, again $-4\lambda_n\varepsilon(-1)^{n_1-1}A_{n_1-1,n_1}\alpha_{n_1-1}>0$
and
the contribution given
these $\alpha_j$, $j_0\leq j\leq j_1$, is 
$$-4\lambda_n\varepsilon\sum_{j=j_0}^{n_1-1}(-1)^jA_{j,n_1}\alpha_j$$
which is also positive by \eqref{A_jincrea}, using a similar care as for the second case.
In particular, the only case when $j_1=n_1-1<j(n_1)$ is when $j(n_1)=n_1$ and that happens only for $n_1=2$, therefore in this case $j_0=j_1=1$ and $n_1=2$.\end{proof}

\medskip

Next, we show that if the monotonicity properties of Lemma~\ref{monotonelemma} are not satisfied, they can be achieved by iterating the following procedure. We begin with a definition.
\begin{defin}
In the first and third case, we say that $j_0< k_0\leq k_1\leq j_1$ define a \emph{strict local maximum region} if $|\alpha_j|=\alpha$ for any $k_0\leq j\leq k_1$ and
$|\tilde{\beta}_{k_0-1}|<\alpha$ and $|\tilde{\beta}_{k_1}|<\alpha$. Note that $k_1<+\infty$ by \eqref{liminf}.
We call $\tilde{\alpha}=\max\{|\tilde{\beta}_{k_0-1}|,|\tilde{\beta}_{k_1}|\}$.

In the second and fourth case, we say that $j_0< k_0\leq k_1\leq j_1$ define a \emph{strict local minimum region} if $|\alpha_j|=\alpha$ for any $k_0\leq j\leq k_1$,
$|\alpha_{k_0-1}|>\alpha$ and either $k_1=j_1$ or $|\alpha_{k_1+1}|>\alpha$.
We call $\tilde{\alpha}=|\alpha_{k_0-1}|$ if $k_1=j_1$
or 
$\tilde{\alpha}=\min\{|\alpha_{k_0-1}|,|\alpha_{k_1+1}|\}$.
\end{defin}

\begin{oss} We satisfy the assumptions of Lemma~\ref{monotonelemma} if and only if there is no local maximum or, respectively, minimum region. For this purpose, when $j_1=+\infty$ we use in particular \eqref{liminf}.
\end{oss}

\begin{lem}
Let $u^1$ be obtained from $u$ by the following procedure.

In the first and third case, for any local maximum region we drop $|\alpha_j|=\alpha$, $k_0\leq j\leq k_1$, to $\tilde{\alpha}$.

In the second and fourth case, for any local minimum region we raise $|\alpha_j|=\alpha$, $k_0\leq j\leq k_1$, to $\tilde{\alpha}$.

Then
$\displaystyle \frac{\partial F_n}{\partial u}(u_n)\geq \frac{\partial F_n}{\partial u^1}(u_n)$,
directional derivatives defined as in \eqref{dirdernew}.
\end{lem}

\begin{proof} We use the same kind of argument as in the proof of first two cases in Lemma~\ref{monotonelemma}. In fact, we call $w^1=u^1+\tilde{v}(u^1)$ and note
that
$$\|u_n+\varepsilon w\|_{BV(\mathbb{R})}=\|u_n+\varepsilon w^1\|_{BV(\mathbb{R})}+2\varepsilon|\alpha-\tilde{\alpha}|.$$ Then for $k_0\leq j\leq k_1$, the values of $\alpha_j$ are constant so satisfy the required monotonicity property.\end{proof}

\medskip

By applying the same procedure to $u^1$, we construct a sequence $\{u^m\}_{m\in\mathbb{N}}$, with coefficients $\alpha_j^m$, $j\in\mathbb{N}$, along which the directional derivative is decreasing.
We note that $\alpha^m_{j_0}=\alpha_{j_0}$ for any $m\in\mathbb{N}$.
More importantly, for any of the regions $[j_0,j_1]$, if for some $u^m$ the coefficients $|\alpha^m_j|$, $j_0\leq j\leq j_1$, are decreasing or, respectively, increasing, then
 the procedure stops in this region, that is, $\alpha^{m+1}_j=\alpha^{m}_j$, for any $j_0\leq j\leq j_1$. Moreover, if $j_1$ is finite, the procedure has to stop after at most $j_1-j_0$ steps.
 In fact, there are at most $j_1-j_0$ intervals which are candidates to be $[k_0,k_1]$ and at each step the number of candidates reduces of at least one, hence the number of candidates is zero after 
 $j_1-j_0$ steps.
 
 After $n_1$ steps, for all regions $[j_0,j_1]$ in the second and fourth cases, the coefficients $|\alpha^{n_1}_j|$, $j_0\leq j\leq j_1$, are increasing, then we can apply Lemma~\ref{monotonelemma}, thus without loss of generality we can assume from the very beginning that $\alpha_j=0$ for any $j=1,\ldots,n_1$.

Let us consider the first and third cases. First of all we note that for any $j>n_1$ and any $m\in\mathbb{N}$,
$|\alpha_j^{m+1}|\leq|\alpha_j^m|\leq |\alpha_j|$.
Besides $j_0$, there might be other coefficients $\alpha_l$, $j_0<l\leq j_1$ which remain constant along the sequence $\{u^m\}_{m\in\mathbb{N}}$. In particular, this is true if
$|\alpha_j|\geq |\alpha_l|$ for any $j_0\leq j\leq l$. When $j_1=+\infty$, there exists a sequence $\{l_k\}_{k\in\mathbb{N}}$ such that  for any $k\in\mathbb{N}$ we have
$j_0<l_k<l_{k+1}$, $|\alpha_{j_0}|>|\alpha_{l_k}|>|\alpha_{l_{k+1}}|$ and
$|\alpha_j|\geq |\alpha_{l_k}|$ for any $j_0\leq j\leq l_k$. This may be easily constructed by induction using \eqref{liminf}. In fact, let $\tilde{l}_1>j_0$ be such that
$|\alpha_{j_0}|>|\alpha_{\tilde{l}_1}|$ and call $l_1$ the index minimizing $|\alpha_j|$ for $j_0\leq j\leq \tilde{l}_1$. Then, let 
$\tilde{l}_2>\tilde{l}_1$ be such that
$|\alpha_{l_1}|>|\alpha_{\tilde{l}_2}|$ and call $l_2$ the index minimizing $|\alpha_j|$ for $j_0\leq j\leq \tilde{l}_2$. Clearly $l_2>l_1$ and by iterating the procedure we construct the desired sequence. We conclude that for any $k\in\mathbb{N}$ and any $m\in\mathbb{N}$, we have
$\alpha_{l_k}^m=\alpha_{l_k}$. We can now conclude the proof of Theorem~\ref{examplethmthird}.

\begin{proof}[Proof of Theorem~\ref{examplethmthird}.]
We assume $\alpha_j=0$ for any $1\leq j\leq n_1$.
We easily infer that
there exists $\{\alpha_j^{\infty}\}_{j\in\mathbb{N}}$ such that,
as $m\to+\infty$, we have that $\alpha_j^m\to\alpha_j^{\infty}$ for any $j\in\mathbb{N}$ and, by the 
dominated convergence theorem, $u^m$ converges
 in $L^2(\R)$ to $u^{\infty}$, the function corresponding to the sequence $\{\alpha_j^{\infty}\}_{j\in\mathbb{N}}$. Then, for $m\in \mathbb{N}\cup\{+\infty\}$,
we call $w^m=u^m+\tilde{v}(u^m)$.
It is easily seen that, for any $m\in\mathbb{N}$,
$$\|u_n+\varepsilon w^{m+1}\|_{BV(\mathbb{R})} \leq \|u_n+\varepsilon w^m\|_{BV(\mathbb{R})}\leq \|u_n+\varepsilon w\|_{BV(\mathbb{R})},$$
therefore by semicontinuity of the total variation
\begin{equation}
\|u_n+\varepsilon w^{\infty}\|_{BV(\mathbb{R})}\leq  \|u_n+\varepsilon w^m\|_{BV(\mathbb{R})}\leq \|u_n+\varepsilon w\|_{BV(\mathbb{R})}.
\end{equation}
By the continuity of $\Lambda$, it is also easy to infer that as $m\to +\infty$
$$-2\lambda_n\varepsilon\langle\tilde{\Lambda}-\Lambda(\sigma_n),\Lambda(u^m)\rangle \to 
-2\lambda_n\varepsilon\langle\tilde{\Lambda}-\Lambda(\sigma_n),\Lambda(u^{\infty})\rangle.$$
Hence, we conclude that for any $m\in\mathbb{N}$
$$\displaystyle \frac{\partial F_n}{\partial u^{\infty}}(u_n)\leq 
 \frac{\partial F_n}{\partial u^m}(u_n)\leq  \frac{\partial F_n}{\partial u}(u_n),$$
where all these directional derivatives are defined as in \eqref{dirdernew}.

Finally, for all intervals $[j_0,j_1]$, $\{\alpha_j^{\infty}\}_{j\in\mathbb{N}}$ satisfies the monotonicity property of Lemma~\ref{monotonelemma}. In fact, whenever $j_1$ is finite, this is true for any $m\geq j_1-j_0$, hence for $m=+\infty$ as well. If $j_1=+\infty$, for any $k\in\mathbb{N}$, for any $m\geq l_k$ and any $j_0\leq j\leq l_k$, we have that $|\alpha_j^m|$ is constant
with respect to $m$ and decreasing with respect to $j$. Hence $|\alpha_j^{\infty}|$ is decreasing with respect to $j$ on $[j_0,l_k]$ for any $k\in \mathbb{N}$, hence on $[j_0,+\infty)$.

By Lemma~\ref{monotonelemma}, since $\widehat{u^{\infty}}=0$, we infer that
$\displaystyle \frac{\partial F_n}{\partial u^{\infty}}(u_n)\geq 0$ and the proof is concluded.\end{proof}

\section{Conclusions}
Our results suggest that for the multiscale procedure for inverse problems it might be convenient to adopt a Hilbert norm regularization, possibly also in the nonlinear case. Such kind of regularization might help both to guarantee convergence in the unknowns space and to stabilize the reconstruction procedure, which is an important desirable feature. In fact, the multiscale procedure for inverse problems might serve for two purposes. First, we solve the inverse problem and at the same time we have a multiscale decomposition of the solution which is driven by the inverse problem itself. Second, as $\lambda_n$ grows, the stability of the minimization problem
$$\min\left\{\lambda_n\big(d(\tilde{\Lambda},\Lambda(u))\big)^{\alpha}+(R(u))^{\beta}:\ u\in E\right\}$$
degrades, sometimes very rapidly. It can be speculated that \eqref{inductiveconstr} is instead more stable, by exploiting the fact 
that $\sigma_{n-1}$ is already a good approximation of the solution, thus we have a very nice initial guess and we can restrict our problem to a local one near $\sigma_{n-1}$. Hence, at the same level of resolution, the multiscale procedure might turn out to be more reliable.%, even if more computationally demanding.

On the other hand, if a different kind of regularization is used, our counterexamples suggest that it would be better to adopt the tighter multiscale procedure developed in \cite{M-N-R} instead of the classical one.

\bigskip

\noindent
\textbf{Acknowledgements.}
SR and LR are supported by the Italian MUR through the PRIN 2022 project “Inverse problems in PDE: theoretical and numerical analysis”, project code: 2022B32J5C (CUP B53D23009200006 and CUP F53D23002710006), under the National Recovery and Resilience Plan (PNRR), Italy, Mission 04 Component 2 Investment 1.1 funded by the European Commission - NextGeneration EU programme. SR is a member of the INdAM research group GNCS, which is kindly acknowledged. LR is also supported by the INdAM research group GNAMPA through 2025 projects.
Part of this work was done during a visit of LR to the Universit\`a di Modena e Reggio Emilia, whose kind hospitality is gratefully acknowledged.

%%%%%%%%%%%%%%%%%%%%%%%%%%%%%%%%%%%%%%%%%%%%%%%%%

\end{document}